\documentclass[11pt]{amsart}

\usepackage{amsfonts, amstext, amsmath, amsthm, amscd, amssymb}
\usepackage[utf8]{inputenc}
\usepackage{float}
\usepackage{graphicx, color}
\usepackage{xcolor}
\usepackage{bm}
\usepackage{microtype}
\usepackage[normalem]{ulem} 

\usepackage[hidelinks,pagebackref,pdftex]{hyperref}
\usepackage{tikz}
\usepackage{tikz-cd}
\usepackage{pgfplots}
\usepackage{subcaption}

\newcommand{\rme}{\mathrm{e}}
\newcommand{\rmi}{\mathrm{i}}
\newcommand{\re}{\operatorname{Re}}
\newcommand{\im}{\operatorname{Im}}

\renewcommand{\setminus}{{\smallsetminus}}

\newcommand{\Addresses}{{% 
\bigskip
  \tiny %

  \noindent  \textsc{R.~N.~Araújo dos Santos} \\
  \textsc{Institute of Mathematics and Computer Science (ICMC)\\
University of São Paulo (USP), Avenida Trabalhador São-Carlense, 400 - Centro\\
CEP: 13566-590 - São Carlos - SP, Brazil.}\\
 \textit{E-mail}:
 \texttt{rnonato@icmc.usp.br}
 \par\nopagebreak
\medskip
  \tiny %

  \noindent \textsc{B.~Bode} \\ 
  \textsc{Departamento de Matemática Aplicada a la Ingeniería Industrial, ETSIDI, Universidad Politécnica de Madrid, \\ 
  Ronda de Valencia 3, 28012 Madrid, Spain.}\\
  \textit{E-mail}: \texttt{benjamin.bode@upm.es}
  \par\nopagebreak
  
\medskip
  
  \tiny %
  
\noindent \textsc{T.~ de Paiva} \\ 
  \noindent \textsc{Beijing International Center for Mathematical Research, Peking University, \\ 
  Beijing 100871, China P.R.} \\
  \textit{E-mail}: \texttt{thhiagodepaiva@gmail.com}

\medskip
  \tiny %

    \noindent \textsc{E.~L.~Sanchez Quiceno} \\ 
  \textsc{Departamento de Matemática, Universidade Federal de São Carlos, \\ 
  Rodovia Washington Luís, Km 235, CP 13560-905, São Carlos - SP, Brazil.}

 \smallskip % 

  \noindent \textsc{Instituto Universitario de Matemática Pura y Aplicada,  Universitat Politècnica de València, \\ 
  Ed. 8E, Camino de Vera, s/n 46022 València, Spain.} \\
  \textit{E-mail}: \texttt{ederleansanchez@alumni.usp.br}
  \par\nopagebreak
  
}}

\newtheorem{theorem}{Theorem}[section]
\newtheorem{proposition}[theorem]{Proposition}
\newtheorem{lemma}[theorem]{Lemma}
\newtheorem{corollary}[theorem]{Corollary}

\newtheorem*{namedtheorem}{\theoremname}
\newcommand{\theoremname}{testing}

\theoremstyle{definition}
\newtheorem{definition}[theorem]{Definition}
\newtheorem{example}[theorem]{Example}

\newtheorem{remark}[theorem]{Remark}

\title[Essential Torus, Link of Mixed Singularity, and Newton Polygon]{Essential Tori associated with links of Mixed Singularities}
\author[R. Araújo dos Santos]{Raimundo N. Araújo dos Santos}
\author[B. Bode]{Benjamin Bode}
\author[T. de Paiva]{Thiago de Paiva}
\author[E. Sanchez Quiceno]{Eder L. Sanchez Quiceno}
\begin{document}

\begin{abstract}
 We establish a direct connection between the analytic data of weakly isolated mixed singularities and the topology of their associated links. More precisely, we prove that the existence of essential tori, topological information, in the complements of links arising from weakly isolated mixed singularities can be detected directly from properties of the defining mixed polynomial, provided that it is convenient, non-degenerate and $\Gamma$-nice.

Our results provide explicit and computable criteria, expressed purely in terms of the polynomial data, that determine the presence of essential tori in the link exterior. In particular, these criteria yield effective conditions ensuring that such links are non-hyperbolic.

This approach provides a new method to extract topological information about link complements without requiring an explicit determination of the link type, thereby establishing a concrete bridge between the analytic structure of mixed polynomials and the geometric topology of their associated links.

\end{abstract}

\maketitle

\section{Introduction}

\subsection{Motivation}
Since the pioneering work of Mostow, Prasad and Thurston \cite{mostow1973strong, prasad1973strong, Thurston}, it has become clear that the geometric structures that can be placed on a topological 3-manifold play a fundamental role in understanding its topology. For instance, many links $L\subset S^3$ admit a metric of constant curvature $-1$ in their complements $S^3\setminus L$ and the volume of $S^3\setminus L$ with respect to that metric is a powerful link invariant, the hyperbolic volume of $L$.
In fact, it follows from Thurston's classification that every knot belongs to exactly one of three families: hyperbolic knots (permitting a hyperbolic metric as above), torus knots or a satellite knots. Precise definitions will be given later, see Subsection~\ref{subsection-Essential}. Analogous statements hold for links, where hyperbolicity can be detected from topological information obtained from the non-existence of certain types of surfaces, called essential surfaces. 

An embedded torus in the complement of a link is called \emph{essential} if it is incompressible and not boundary-parallel, see Definition~\ref{essential} for a formal definition. The presence of such a surface reflects non-trivial topological structure in the link complement and captures non-trivial topological features of the link complement. In particular, it follows from Thurston’s result that if the complement of a link contains an essential torus, then the link is not hyperbolic.

From this point of view, the existence of an essential torus indicates that the link exhibits a more intricate topological structure, typically associated with satellite-type behavior. More generally, essential tori encode structural information about the link complement and reveal the presence of non-trivial decompositions. Thus, detecting essential tori can be viewed as a way to extract qualitative topological information about the link, going beyond the mere distinction between hyperbolic and non-hyperbolic cases.

There is a substantial body of work devoted to understanding when link complements admit hyperbolic structures. In particular, significant progress has been made in detecting hyperbolicity using geometric and topological methods; see, for instance, the book by Purcell~\cite{Purcell2020}. These approaches typically rely on detailed analysis of the topology or geometry of the link complement.

Knots and links arise naturally in many areas of mathematics. In this paper, we focus on links associated with singularities of real polynomial maps. A central problem in this context is to understand how the analytic structure of a defining polynomial map influences the topology of the associated link and its complement.

Let $f:\mathbb{C}^2 \to\mathbb{C}$ be a holomorphic function, defining a complex plane curve with an isolated singularity at the origin in $\mathbb{C}^2.$ Using classical tools from singularity theory one may consider $f$ as a holomorphic polynomial, after a biholormorphic coordinate change at the origin. Hence, one may use arguments of transversality to show that for all sufficiently small radii $\rho$ the intersection $L_f:=V(f)\cap S^3_{\rho}$ of the complex plane curve $V(f)=\{(z_1,z_2)\in\mathbb{C}^2:f(z_1,z_2)=0\}$ and the 3-sphere $S^3_\rho$ of radius $\rho$ centered at the singularity produces a topological link, in the sense that the space $L_f$ is a finite disjoint union of circles $S^{1}$ embedded in $S^3_\rho.$ One can show that the isotopy class of $L_f$ in $S^3_\rho$ does not depend on the choice of the radius $\rho$, as long as it is chosen small enough. Therefore, the link type of $L_f$ becomes a local topological description of the singularity of $f$. 
%Moreover, it was proved by J. Milnor in \cite{Milnor1968} that 
%the argument function $\mathrm{arg}(f):=\dfrac{f}{|f|}:S^3_\rho \setminus L_{f}\to S^{1}$ is a smooth projection of a locally trivial fiber bundle. Hence,  all 
%such links are fibered links and 
These links are classically called  \textit{algebraic links}. We sometimes simply say $L_f$ is the link of $f$. The link types that can arise in this way are completely characterized. They are torus links or certain iterated cables of torus links (a special type of satellite link), or possibly certain unions of such links (if $f$ is not irreducible). In particular, none of these algebraic links is hyperbolic. However, this scenario changes drastically if we  consider the real counterparts of holomorphic polynomials, namely real polynomial maps $f:\mathbb{R}^{4}\to \mathbb{R}^{2}$. One finds a much richer topological landscape where the constraints of the complex setting no longer hold. Indeed, in \cite{perron}, B. Perron proved the existence of a real polynomial map $g:\mathbb{R}^{4}\to \mathbb{R}^{2},$ with isolated singularity  at the origin such that for all $\rho>0$ small enough the link $L_{g}=V(g)\cap S_{\rho}^{3}$ is isotopic to figure-eight knot, which is the canonical example of a hyperbolic knot.
 
 We should mention an important distinction to the complex case. A singularity
of a real map $f$ can be \textit{weakly isolated}, that is, it is the only singular point of $f$ in
$U \cap V(f)$ for some neighbourhood $U$, or it can be isolated, that is, it is the
only singular point in $U$ for some neighbourhood $U$. In both cases, one may use again arguments of transversality to show that the link of a weakly isolated singularity is well defined, up to isotopy, on all spheres of radii small enough centered at the origin. In fact, it was proved in \cite{akbulutking} that every link arises as the link of a weakly isolated singularity be it hyperbolic or not, while links of isolated singularities must be fibered, according to Milnor's work in \cite{Milnor1968}. But obviously isolated singularities pose much stronger conditions on the polynomial realization.

This paves the way to an interesting question, as:

\vspace{0.2cm}

\textit{To what extent can analytic properties of a defining real polynomial map determine topological features of the associated link and its complement?}

\vspace{0.2cm}

The present paper addresses this problem by focusing on a specific class of real maps: those arising from the class of convenient, non-degenerate, $\Gamma$-nice mixed polynomials, see Definition~\ref{def:convandnd} and \ref{def:Gammaniceness}, which are already known to have weakly isolated singularities. We employ geometric and combinatorial tools extracted from a 2-dimensional Newton polygon to detect topological information from these links. More precisely, we show that the existence of essential tori in the link complement can be detected directly from the analytic data of the defining mixed polynomial, thereby obtaining concrete topological information directly from the underlying analytic structure. This provides explicit and computable criteria, expressed in terms of the Newton polygon, that allow one to determine non-hyperbolicity without requiring an explicit description of the link type.

\subsection{Main Results}
Precise definitions and notation regarding the Newton polygon of a mixed polynomial are reviewed in Section~\ref{sec:Newton}. We can write every real polynomial $f:\mathbb{R}^4\to\mathbb{R}^2$ as a mixed polynomial $f:\mathbb{C}^2\to\mathbb{C}$ and associate to each such polynomial its Newton polygon. If $f$ satisfies several desirable properties (convenient, non-degenerate and $\Gamma$-nice), then it has a weakly isolated singularity at the origin and its link can be described as the union of (possibly empty) links $L_1, L_2,\ldots, L_N$, each of which is associated with a compact 1-face ``edge") of the boundary of the Newton polygon of $f$. If we label the edges in the Newton diagram from left to right by $1,2,\ldots,N$, then $L_i$ is associated with the edge $i$. Furthermore, there are natural tori that separate the different components. We write $\partial V_i$ for the torus that separates $\bigcup_{j=1}^iL_j$ from $\bigcup_{j=i+1}^N L_j$.

We write \(\mathcal{P}_0(f)=\{\Delta_0, \Delta_1, \dots, \Delta_N\}\) for the $0$-faces (``vertices") of the Newton boundary of $f$, so that $\Delta_i$ is the vertex that is the common point of the edges $i$ and $i+1$. For each such vertex $\Delta_i$ we can define certain multiplicities $\mathrm{m}_\mathrm{s}(f^t_{\Delta_i},0)$ and $\mathrm{m}_\mathrm{s}(f^\varphi_{\Delta_i},0)$, which can be calculated directly from the appropriate monomials of $f$.

Our main results can then be stated as follows. 
\begin{theorem}\label{prop:fastcriterion}
Let \(f:\mathbb{C}^2 \to \mathbb{C}\) be a convenient and non-degenerate mixed polynomial that is \(\Gamma\)-nice. Fix \(i \in \{1,2,\dots,N-1\}\). Denote by $L_f$ the link of $f$. If
\[
\min \Bigl\{ \bigl|\mathrm{m}_{\mathrm{s}}(f^{t}_{\Delta_i},0)\bigr|, 
\bigl|\mathrm{m}_{\mathrm{s}}(f^{\varphi}_{\Delta_i},0)\bigr| \Bigr\} > 1,
\]
then \(\partial V_i\) is an essential torus in the exterior of $L_f$. Thus, $L_f$ has at least one essential torus. In particular, $L_f$ is not hyperbolic.
\end{theorem}

\begin{theorem}\label{thm:geralcriterionessentialtori}
Let $f:\mathbb{C}^2 \to \mathbb{C}$ be a convenient and non-degenerate mixed polynomial that is $\Gamma$-nice. Let $I_f=\{i_1,i_2,\dots,i_{n}\}$ be the set of 1-faces of the Newton boundary of $f$ such that $L_{i_j}$ is not empty for all $j$ and assume $n\geq 3$. Denote by $L_f$ the link of $f$.

\medskip
\noindent (i) If $|\mathrm{m}_{\mathrm{s}}(f^{t}_{\Delta_{i_1}},0)|>1$ and for some $i_j>i_1$, $$|\mathrm{m}_{\mathrm{s}}(f^\varphi_{\Delta_{i_j-1}},0)-\mathrm{m}_{\mathrm{s}}(f^\varphi_{\Delta_{i_j}},0)|>0,$$ then $\partial V_{i_1}$ is essential in the exterior of $L_f$. 
    
     \medskip
\noindent
(ii) If $1<k<n-1$ and for $i_j\leq i_k<i_l$, $$\min(|\mathrm{m}_{\mathrm{s}}(f^{t}_{\Delta_{i_j}},0)-\mathrm{m}_{\mathrm{s}}(f^{t}_{\Delta_{i_j-1}},0)|,|\mathrm{m}_{\mathrm{s}}(f^\varphi_{\Delta_{i_l-1}},0)-\mathrm{m}_{\mathrm{s}}(f^\varphi_{\Delta_{i_l}},0)|)>0,$$ 
then $\partial V_{i_k}$ is essential in the exterior of $L_f$.

    \medskip
\noindent
(iii) If $|\mathrm{m}_{\mathrm{s}}(f^\varphi_{\Delta_{i_n-1}},0)|\}>1$ and for some $i_j<i_n$, $$|\mathrm{m}_{\mathrm{s}}(f^{t}_{\Delta_{i_j}},0)-\mathrm{m}_{\mathrm{s}}(f^{t}_{\Delta_{i_j-1}},0)|>0,$$ then $\partial V_{i_{n-1}}$ is essential in the exterior of $L_f$.

Thus, in any of these cases, $L_f$ has at least one essential torus. In particular, $L_f$ is not hyperbolic.
\end{theorem}

Theorem~\ref{prop:fastcriterion} serves as a faster criterion, as it relies solely on the analysis of the multiplicity values associated with a single $0$-face $\Delta_i$.
However, Theorem~\ref{prop:fastcriterion} is less effective than Theorem~\ref{thm:geralcriterionessentialtori} in general. This is illustrated in Example~\ref{ex:example2}, where the existence of an essential torus cannot be detected by the first criterion but is guaranteed by applying Theorem~\ref{thm:geralcriterionessentialtori}.

Theorem~\ref{thm:geralcriterionessentialtori} provides more general criteria than Theorem~\ref{prop:fastcriterion}, as it covers cases with lower multiplicities. However, it requires the Newton boundary decomposition to have at least three 1-faces whose corresponding links $L_i$ are not the empty set ($n\geq 3$).  

Beyond the criteria for the existence of essential tori provided by Theorems~\ref{prop:fastcriterion} and \ref{thm:geralcriterionessentialtori}, our techniques establish a  complementary condition for non-hyperbolicity. This approach accounts implicitly not only for the presence of essential tori but also for the existence of essential spheres in the link complement. 
 \begin{theorem}
    \label{thm:nonemptycriterion}
Let \(f:\mathbb{C}^2 \to \mathbb{C}\) be a convenient and non-degenerate mixed polynomial that is $\Gamma$-nice. If there exist four distinct indexes $ i\in \{1,2,\dots,N\}$ satisfying 
\[
\max\bigl\{|\mathrm{m}_{\mathrm{s}}(f^t_{\Delta_i}, 0) - \mathrm{m}_{\mathrm{s}}(f^t_{\Delta_{i-1}}, 0)|,\,|\mathrm{m}_{\mathrm{s}}(f^\varphi_{\Delta_{i-1}}, 0) - \mathrm{m}_{\mathrm{s}}(f^\varphi_{\Delta_i}, 0)|\bigr\} > 0,
\]
then \(L_f\) is not hyperbolic.
\end{theorem} 
We illustrate in Example~\ref{ex:nonhipn>3} that Theorem~\ref{thm:nonemptycriterion} identifies non-hyperbolic links of mixed polynomials not covered by Theorems~\ref{prop:fastcriterion} and~\ref{thm:geralcriterionessentialtori}.

\subsection{Structure of the article}
The outline of the remainder of the article is as follows. In Section~\ref{sec:background} we review the necessary background and definitions, both regarding hyperbolic knots and links, and the description of links of mixed polynomials in terms of Newton polygons. The tori that decompose the links of singularities $L_f$ into various components form a sequence of nested tori. In Section~\ref{sec:essential} we study in general, when such surfaces are essential (independent of the context of singularities). Section~\ref{sec:results} establishes the criteria for the existence of essential tori, which in turn determine the non-hyperbolicity of the links of mixed singularities, thereby proving our main theorems. Explicit calculations and examples are provided in Section~\ref{sec:examples}.

\vspace{0.2cm}

\textbf{Acknowledgments:} 
The first author would like to thank the partial supports of the Fapesp thematic-project
``Novas fronteiras na Teoria de Singularidades,'' proc.\ 2019/21181-0, and the
CNPq-Universal project ``Novas tend\^{e}ncias no estudo da topologia de
singularidades reais,'' proc.\ 408147/2023-7. The fourth author was supported by FAPESP (grants 2023/11366-8 and 2024/17116-6). 
 
\section{Background and Definitions}\label{sec:background}

The main goal of this paper is to detect the existence of essential tori in the exterior of links arising from mixed polynomials. Since the presence of such tori is both an obstruction to hyperbolicity and a source of topological information about the link complement, we review the necessary notions from 3--manifold topology, with emphasis on essential tori.

\subsection{Essential Tori and Hyperbolicity}\label{subsection-Essential}

Let $L \subset S^3$ be a link and denote by 
\[
M_L := S^3 \setminus N(L)
\]
its exterior, where $N(L)$ is an open tubular neighbourhood of $L$.

\begin{definition}
Let $T \subset M_L$ be an embedded torus in a link exterior. An embedded disk $D \subset M_L$ with $\partial D \subset T$ is called a \emph{compression disk} for $T$ if $\partial D$ does not bound a disk in $T$. The torus $T$ is called \emph{compressible} if it admits a compression disk, and \emph{incompressible} otherwise.
\end{definition}

\begin{definition}\label{essential}
An embedded torus $T \subset M_L$ is called \emph{essential} if it is incompressible and not isotopic to a component of $\partial M_L$.
\end{definition}

Essential tori play a central role in the study of link exteriors, as they detect non-trivial topological structure.

A link exterior is called \emph{atoroidal} if it contains no essential tori.

A knot $K \subset S^3$ is called a \emph{satellite knot} if its exterior contains an essential torus. A \emph{torus knot} is a knot that can be embedded on the surface of an unknotted torus in $S^3$.

\begin{definition}
A link $L \subset S^3$ is called \emph{hyperbolic} if its exterior $M_L$ admits a complete Riemannian metric of constant curvature $-1$ and finite volume.
\end{definition}

A fundamental consequence of Thurston's work is that hyperbolicity admits a purely topological characterization.

\begin{theorem}[Thurston \cite{Thurston1982}] 
Every knot in $S^3$ is either hyperbolic, a torus knot, or a satellite knot.
\end{theorem}

More generally, hyperbolicity of 3--manifolds with torus boundary can be characterized in purely topological terms.

\begin{theorem}[Hyperbolization Theorem {\cite{Thurston1982}}]\label{Thurston}
Let $M$ be a compact, orientable 3--manifold with nonempty torus boundary. Then the interior of $M$ admits a complete hyperbolic metric if and only if $M$ is irreducible, boundary-irreducible, atoroidal, and anannular.
\end{theorem}

We refer to \cite[Chapter 8]{Purcell2020} for the definitions of irreducible, boundary-irreducible, and anannular.

In particular, we obtain the key implication used throughout this article.

\begin{corollary}\label{cor:essential_torus_obstruction}
If the exterior $M_L$ of a link $L \subset S^3$ contains an essential torus, then $L$ is not hyperbolic.
\end{corollary}

Thus, detecting essential tori provides an effective and conceptually simple method to prove non-hyperbolicity.

\subsection{JSJ Decomposition and Topological Structure}

Essential tori play a fundamental role in the JSJ decomposition of 3--manifolds. Roughly speaking, the JSJ decomposition expresses a compact, orientable, irreducible 3--manifold as a union of simpler pieces separated by a minimal collection of disjoint essential tori.

In the case of link complements, these tori encode the internal topological structure of the manifold and detect non-trivial decompositions into simpler components.

\begin{remark}
Although essential tori form the building blocks of the JSJ decomposition, in this paper we do not attempt to determine the full JSJ decomposition of the link complement. Instead, our goal is to provide effective criteria to detect the existence of essential tori.

From this point of view, the existence of an essential torus indicates that the link is not hyperbolic and reflects a more intricate topological structure, typically associated with satellite-type behavior.

More generally, essential tori encode structural information about the link complement and reveal the presence of non-trivial decompositions. Thus, detecting essential tori can be viewed as a way to extract qualitative topological information about the link, going beyond the mere distinction between hyperbolic and non-hyperbolic cases.
\end{remark}
\subsection{Mixed polynomials}
A \textit{mixed polynomial} in two complex variables $z=(u, v)$ is a formal expression $f(u, \bar{u}, v, \bar{v})$ that admits a polynomial expansion in both the complex variables $z = (u, v)$ and their conjugates $\bar{z} = (\bar{u}, \bar{v})$. Specifically, it is given by the following sum:
\begin{equation}
f(u, \bar{u}, v, \bar{v}) = \sum_{\boldsymbol{\nu}, \boldsymbol{\mu}} c_{\boldsymbol{\nu}, \boldsymbol{\mu}} z^{\boldsymbol{\nu}} \bar{z}^{\boldsymbol{\mu}}, \quad c_{\boldsymbol{\nu}, \boldsymbol{\mu}} \in \mathbb{C}
\end{equation}
where $\boldsymbol{\nu} = (\nu_1, \nu_2)$ and $\boldsymbol{\mu} = (\mu_1, \mu_2)$ are multi-indices in $\mathbb{N}^2$. Throughout this work, we assume that the mixed polynomial is normalized such that $c_{\bm0,\bm0} = 0$.
For the sake of simplicity, we shall employ the same symbol $f$, or simply the notation $f(u,v)$, to refer to both the mixed polynomial expression and its associated complex-valued function $f: \mathbb{C}^2 \to \mathbb{C}$ defined by evaluation $(u,v) \mapsto f(u,\bar{u},v,\bar{v})$, leaving the dependence on complex conjugates implicit unless required for clarity. By identifying $\mathbb{C}^2$ with $\mathbb{R}^4$ via the coordinates $u = x_1 + \rmi x_2$ and $v = x_3 + \rmi x_4$, the function $f$ induces a real analytic map $\tilde{f}: \mathbb{R}^4 \to \mathbb{R}^2$ defined by:
\begin{equation}
\bm x = (x_1, x_2, x_3, x_4) \mapsto \tilde{f}(\bm x) = \left( \re (f(u, \bar{u}, v, \bar{v})), \im(f(u, \bar{u}, v, \bar{v})) \right).
\end{equation}

The mixed variety is defined as:
\begin{equation}
V(f) := \{ (u, v) \in \mathbb{C}^2 : \tilde{f}(\bm x) = 0 \}.
\end{equation}

The \textit{singular set} $\Sigma(f)$ of the mixed polynomial function $f$ is defined as the set of points where the induced map $\tilde{f}$ fails to be a submersion:
\begin{equation}
\Sigma(f) := \{ (u, v) \in \mathbb{C}^2 : \mathrm{rank}(J\tilde{f}(\bm x)) < 2 \}
\end{equation}
where $J \tilde{f}(\bm x)$ denotes the Jacobian matrix of $\tilde{f}$.

\subsection{The Newton polygon}\label{sec:Newton}

The \textit{support} of a mixed polynomial $f$ is defined, in \cite{Oka2010},  by 
\[\mathrm{supp} (f):=\{\boldsymbol{\nu}+\boldsymbol{\mu} \in (\mathbb{Z}_{\geq 0})^2 \mid c_{\boldsymbol{\nu},\boldsymbol{\mu}}\neq 0 \}.\]

The \textit{Newton polygon} of $f$ is defined as the convex hull of
\[ \bigcup_{ \bm{\omega} \in \rm{supp}(f)} \{ \bm{\omega} + (\mathbb{R}_{\geq 0})^2\} ,\]
where + denotes the Minkowski sum. The \textit{Newton boundary}  $\Gamma(f)$ is defined as the union of compact faces of $\Gamma_+(f)$. For a face $\Delta$ of $\Gamma(f)$ we define 
\begin{equation*}
	f_{\Delta}(u, \bar{u}, v, \bar{v}) := \sum_{\boldsymbol{\nu}+\boldsymbol{\mu} \in \Delta} c_{\boldsymbol{\nu},\boldsymbol{\mu}} z^{\boldsymbol{\nu}} \bar{z}^{\boldsymbol{\mu}}.
\end{equation*}
\begin{definition}[{\cite{Oka2010}}]\label{def:convandnd}
Let $f:\mathbb{C}^2 \to \mathbb{C}$ be a mixed polynomial with Newton polygon $\Gamma(f)$.
\begin{itemize}
    \item We say that $f$ is \textit{convenient} if $\Gamma(f)$ intersects all coordinate axes.
    \item 
We say that $f$ is \textit{Newton non-degenerate} (or simply \textit{non-degenerate}) if, for every face $\Delta$ in $\Gamma(f)$, the following condition is satisfied:
$$\Sigma(f_\Delta) \cap V(f_\Delta) \cap (\mathbb{C}^*)^2 = \emptyset.$$ 
\end{itemize}
\end{definition}
\begin{remark}[Non-degeneracies]\label{rmk:ndconditions}
\hspace{2.0cm} 
\begin{itemize}
    \item[(i)]  Every convenient and non-degenerate mixed polynomial $f$ has a weakly isolated singularity at the origin, that is, $\Sigma(f) \cap V(f) \cap U =  \{0\}$ for some neighbourhood $U$ of the origin \cite{Oka2010}.  
    \item[(ii)] In this paper, we present our results within the set of links of convenient and non-degenerate mixed polynomials. This approach maintains the generality of the settings in \cite{Araujo2024,Bode:part2} regarding inner non-degeneracy, as any such polynomial can be deformed into one that is both convenient and non-degenerate while preserving its link type \cite[Proposition 4.31]{Leal2025}. 
\end{itemize}
\end{remark}
Let $f$ be a mixed polynomial. We define the set $\mathcal{P}(f)$ as the collection of 1-faces (edges) of the Newton boundary $\Gamma(f)$, labeled sequentially as $\Delta^1_1, \dots, \Delta^1_N$ following the Newton boundary from left to right\footnote{Equivalently, $\mathcal{P}(f)$ can be defined as the collection of positive weight vectors $P_i$ associated with each face $\Delta^1_i$, such that $P_i$ is orthogonal to the linear subspace parallel to $\Delta^1_i$ (see~\cite{Araujo2024}).}.  Similarly, the set of 0-faces (vertices) is denoted by $\mathcal{P}_0(f) := \{ \Delta_0, \Delta_1, \dots, \Delta_N \}$, where $\Delta_i := \Delta^1_i \cap \Delta^1_{i+1}$, $i = 1, \dots, N-1$, denote the non-extreme vertices, while $\Delta_0$ and $\Delta_N$ denote the extreme vertices, i.e., the vertices that lie on the coordinate axes.
%{\color{red}Benjamin:I moved the footnote to the end of the sentence. Initially, I thought $^1\mathcal{P}(f)$ was some mathematical notation.}
\begin{definition}[{\cite{Araujo2024}}]\label{def:Gammaniceness}
Let \( f:\mathbb{C}^2 \to \mathbb{C} \) be a mixed polynomial. We say that $f$ is \textit{$\Gamma$-nice} if for every non-extreme vertex $\Delta$ in $\mathcal{P}_0(f)$, 
\( V(f_{\Delta}) \cap (\mathbb{C}^*)^2= \emptyset.\)
\end{definition}
\begin{remark}\label{rmk:ndconditions} 
 In \cite{Bode:part2}, the set of links of a wider class of $\Gamma$-nice mixed polynomials, which includes the convenient and non-degenerate ones, was completely characterized via certain symmetry conditions.
\end{remark}
\begin{remark}[$\Gamma$-niceness]\label{rmk:niceness}
Let $\Delta=(a,b)$ be a vertex of the Newton boundary of $f$, that is, $\Delta \in \mathcal{P}_0(f)$. 
\begin{itemize} 
    \item[(i)] A mixed polynomial $f(u, \bar{u},v,\bar{v})$ is said to be \textit{$x$-semiholomorphic}, for $x \in \{u, \bar{u}, v, \bar{v}\}$, if $f$ is algebraically independent of $\bar{x}$, where $\overline{(\bar{u})}=u$ and $\overline{(\bar{v})}=v$. For instance, if $f$ is $u$-semiholomorphic, then $\mu_1=0$ for any monomial $c_{\nu_1,\nu_2,\mu_1,\mu_2}u^{\nu_1}v^{\nu_2}\bar{u}^{\mu_1}\bar{v}^{\mu_2}$ of $f$ with $c_{\nu_1,\nu_2,\mu_1,\mu_2}\neq 0$. Analogously, being $\bar{u}$-, $v$-, or $\bar{v}$-semiholomorphic implies that $\nu_1$, $\mu_2$, or $\nu_2$ must vanish, respectively. In \cite{Araujo2024} it was proved that if $f_{\Delta}$ is semiholomorphic then \( V(f_\Delta) \cap (\mathbb{C}^*)^2= \emptyset\) if and only if $f_\Delta$ is non-degenerate.  
\item[(ii)] In general, the non-degeneracy of $f_\Delta$ does not guarantee that $V(f_\Delta) \cap (\mathbb{C}^*)^2$ is empty (see \cite{Araujo2024,Leal2025}). 
\item[(iii)] Let $u=R \rme^{\rmi \varphi}$ and $v=r \rme^{\rmi t}$. We have:
$$f_\Delta(R \rme^{\rmi \varphi}, R \rme^{-\rmi \varphi}, r \rme^{\rmi t}, r \rme^{-\rmi t}) = R^a r^b f_\Delta(\rme^{\rmi \varphi}, \rme^{-\rmi \varphi}, \rme^{\rmi t}, \rme^{-\rmi t}).$$
If $V(f_\Delta) \cap (\mathbb{C}^*)^2 = \emptyset$, it follows that the set $\{(\varphi, t) \in [0, 2\pi]^2 : f_\Delta(\rme^{\rmi \varphi}, \rme^{-\rmi \varphi}, \rme^{\rmi t}, \rme^{-\rmi t}) = 0\}$ is empty. Consequently, if $f$ is $\Gamma$-nice, convenient, and non-degenerate, then for any $\Delta_i$ in $\mathcal{P}_0(f)$ the set
$$\{(\varphi, t) \in [0, 2\pi]^2 : f_{\Delta_i}(\rme^{\rmi \varphi}, \rme^{-\rmi \varphi}, \rme^{\rmi t}, \rme^{-\rmi t}) = 0\}$$
is empty. Indeed, for $i=1, \dots, N-1$, this follows directly from the $\Gamma$-niceness condition. For $i=0$ and $i=N$, since $f_{\Delta_0}$ and $f_{\Delta_N}$ are semiholomorphic and non-degenerate, (i) ensures that $V(f_{\Delta_i}) \cap (\mathbb{C}^*)^2$ is empty.
\end{itemize}
\end{remark}
\begin{example}\label{ex:example1}
Consider the convenient and non-degenerate mixed polynomial
\[
f(u,\bar{u},v,\bar{v}) = u^5 + u^2 \bar{u}^2 v + u^3 v^2 - \mathrm{i} u \bar{u}^2 v^2 + u^2 \bar{u} v^2 + \bar{u} v^6 + v^9.
\]  
The support is explicitly given by:
\begin{equation*}
\mathrm{supp}(f) = \{ (5,0), (4,1), (3,2), (1,6), (0,9) \}.
\end{equation*}
\begin{figure}[h!]
 \centering
 \begin{tikzpicture}[scale=0.60]
\begin{axis}[axis lines=middle,axis equal,yticklabels={0,,2,4,6,8,10}, xticklabels={0,,2,4,6,8,10},domain=-10:10,     xmin=0, xmax=10,
                    ymin=0, ymax=9,
                    samples=1000,
                    axis y line=center,
                    axis x line=center]
\addplot coordinates{(11,0)(5,0) (3,2) (1,6) (0,9) (0,10)};
 \filldraw[red] (10,85) node[anchor=north ] {$\Delta^1_1$};
  \filldraw[red] (25,50) node[anchor=north ] {$\Delta^1_2$};
   \filldraw[red] (45,20) node[anchor=north ] {$\Delta^1_3$};
  \filldraw[black] (65,65) node[anchor=north ] {$\Gamma_+(f)$};
\filldraw[blue] (40,10) circle (2pt) node[anchor=south west] {};
\filldraw[blue] (30,20) circle (2pt) node[anchor=south west] {};
\filldraw[blue] (0,90) circle (2pt) node[anchor=south west] {};
\filldraw[blue] (10,60) circle (2pt) node[anchor=south west] {};
\filldraw[blue] (50,0) circle (2pt) node[anchor=south west] {};
\fill[yellow!90,nearly transparent] (0,90) --(10,60)--(30,20) -- (50,0) -- (260,0) -- (260,300) -- (0,300) --cycle;
\end{axis}
\end{tikzpicture}
\caption{The Newton polygon $\Gamma_+(f)$.}
\label{newtonboundaryholomorphicq}
\end{figure}
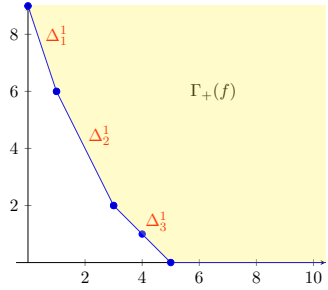
The set of edges of $\Gamma(f)$, see Figure~\ref{newtonboundaryholomorphicq}, is
$\mathcal{P}(f) = \{ \Delta^1_1, \Delta^1_2, \Delta^1_3 \}.$

The face functions are: 
\begin{align*}
f_{\Delta^1_1}(u, \bar{u}, v, \bar{v}) &= \bar{u} v^6 + v^9, \\
f_{\Delta^1_2}(u, \bar{u}, v, \bar{v}) &= (u^3 - \mathrm{i} u \bar{u}^2 + u^2 \bar{u})v^2 + \bar{u} v^6, \\
f_{\Delta^1_3}(u, \bar{u}, v, \bar{v}) &= u^5 + u^2 \bar{u}^2 v + (u^3 - \mathrm{i} u \bar{u}^2 + u^2 \bar{u})v^2. 
\end{align*}
The vertices of the Newton boundary are 
\(\Delta_0 = (0,9), \ \Delta_1 = (1,6), \ \Delta_2 = (3,2)\), and  \(\Delta_3 = (5,0)\), 
with associated face functions:
\begin{align*}
f_{\Delta_0}(u, \bar{u}, v, \bar{v}) &= v^9, \\
f_{\Delta_1}(u, \bar{u}, v, \bar{v}) &= \bar{u} v^6, \\
f_{\Delta_2}(u, \bar{u}, v, \bar{v}) &= (u^3 - \mathrm{i} u \bar{u}^2 + u^2 \bar{u}) v^2, \\
f_{\Delta_3}(u, \bar{u}, v, \bar{v}) &= u^5.
\end{align*}
Since $f_{\Delta_1}$ and $f_{\Delta_2}$ are semiholomorphic, 
by Remark~\ref{rmk:niceness}(i), it follows that $f$ is $\Gamma$-nice. 
\end{example}

\subsection{Nested decomposition of links of mixed singularities}

Let $f$ be a convenient and non-degenerate mixed polynomial that is $\Gamma$-nice.
In the convenient case, the number of 1-faces (denoted by $N$) of $\Gamma(f)$ is non-zero. Then the Newton boundary has $N+1$ vertices, $\Delta_0, \Delta_1, \dots, \Delta_{N}$, ordered so that $\Delta_0$ is contained in the vertical axis and for $i=1,2,\dots,N$, the vertices $\Delta_{i-1}$ and $\Delta_{i}$ bound the same 1-face $\Delta^1_{i} \in  \mathcal{P}(f)$.

Fix $\Delta^1_{i} \in \mathcal{P}(f)$. We define $g_i: \mathbb{C}\times S^1 \to \mathbb{C}$ by:
\begin{equation}
    g_i(u,\rme^{\rmi t})=f_{\Delta^1_i}(u,\bar{u},\rme^{\rmi t},\rme^{-\rmi t}).\label{eq:def_gi}
\end{equation}
For $i=1,2,\dots,N$, define 
$$L_i:= g_i^{-1}(0)\cap (\mathbb{C}^* \times S^1).$$
\begin{lemma}\label{lemma:Li}
  Let \(f:\mathbb{C}^2 \to \mathbb{C}\) be a convenient and non-degenerate mixed polynomial that is \(\Gamma\)-nice. Fix \(i \in \{1,2,\dots,N-1\}\). Then,  
   the set $L_i$ is a compact set in $\mathbb{C}^* \times S^1$. 
\end{lemma}
     \begin{proof}
     Let $u=R\rme^{\rmi \varphi}$. The terms of $f_{\Delta^1_i}(R\rme^{\rmi \varphi}, R\rme^{-\rmi \varphi}, \rme^{\rmi t}, \rme^{-\rmi t})$ with minimal and maximal degrees in $R$ are the terms of $f_{\Delta_{i-1}}(R\rme^{\rmi \varphi}, R\rme^{-\rmi \varphi}, \rme^{\rmi t}, \rme^{-\rmi t})$ and $f_{\Delta_{i}}(R\rme^{\rmi \varphi}, R\rme^{-\rmi \varphi}, \rme^{\rmi t}, \rme^{-\rmi t})$, respectively. Let $\Delta_i = (a_i, b_i)$; then:
    \[
    f_{\Delta_{i}}(R\rme^{\rmi \varphi}, R\rme^{-\rmi \varphi}, \rme^{\rmi t}, \rme^{-\rmi t}) = R^{a_i} f_{\Delta_{i}}(\rme^{\rmi \varphi}, \rme^{-\rmi \varphi}, \rme^{\rmi t}, \rme^{-\rmi t}).
    \]
    According to Remark~\ref{rmk:niceness}(iii), $f_{\Delta_{i}}(\rme^{\rmi \varphi}, \rme^{-\rmi \varphi}, \rme^{\rmi t}, \rme^{-\rmi t})$ does not vanish for any $\varphi$ or $t$. This ensures that the zeros of $g_i$ are contained within $\mathbb{D}_{\varepsilon} \times S^1$, where $\mathbb{D}_\varepsilon$ is a sufficiently large disk centered at the origin. On the other hand, since Remark~\ref{rmk:niceness}(iii) also implies that $f_{\Delta_{i-1}}(\rme^{\rmi \varphi}, \rme^{-\rmi \varphi}, \rme^{\rmi t}, \rme^{-\rmi t})$ is non-vanishing for all $\varphi$ and $t$, it follows that $g_i$ has no zeros in $(\mathbb{D}_{\varepsilon'} \setminus \{0\}) \times S^1$ for a sufficiently small disk $\mathbb{D}_{\varepsilon'}$ centered at the origin. Therefore, we get that $g_i^{-1}(0)\cap (\mathbb{C}^* \times S^1)$ is compact in $\mathbb{C}^* \times S^1$ as desired. 
\end{proof}
\begin{example}\label{Ex:ex2gis}
Let $f$ be the mixed polynomial in Example~\ref{ex:example1}: 
\[
f(u,\bar{u},v,\bar{v}) = u^5 + u^2 \bar{u}^2 v + u^3 v^2 - \mathrm{i} u \bar{u}^2 v^2 + u^2 \bar{u} v^2 + \bar{u} v^6 + v^9.
\]
The maps $g_i$ associated with each 1-face are:
\begin{align*}
g_ 1(u, \rme^{\rmi t}) &= \bar{u} \rme^{6\rmi t} + \rme^{9\rmi t}, \\
g_2(u, \rme^{\rmi t}) &= (u^3 - \mathrm{i} u \bar{u}^2 + u^2 \bar{u}) \rme^{2\rmi t} + \bar{u} \rme^{6 \rmi t}, \\
g_3(u, \rme^{\rmi t}) &= u^5 + u^2 \bar{u}^2 \rme^{\rmi t} + (u^3 - \mathrm{i} u \bar{u}^2 + u^2 \bar{u}) \rme^{2\rmi t}. 
\end{align*}
 We present the illustrations in Figure~\ref{fig:linksofgis} for the link $L_i$ ($i=1, 2, 3$), which successfully capture the essential topological properties sought in our analysis.

\begin{figure}[h]
    \centering
    \begin{subfigure}[b]{0.30\textwidth}
        \includegraphics[width=\textwidth]{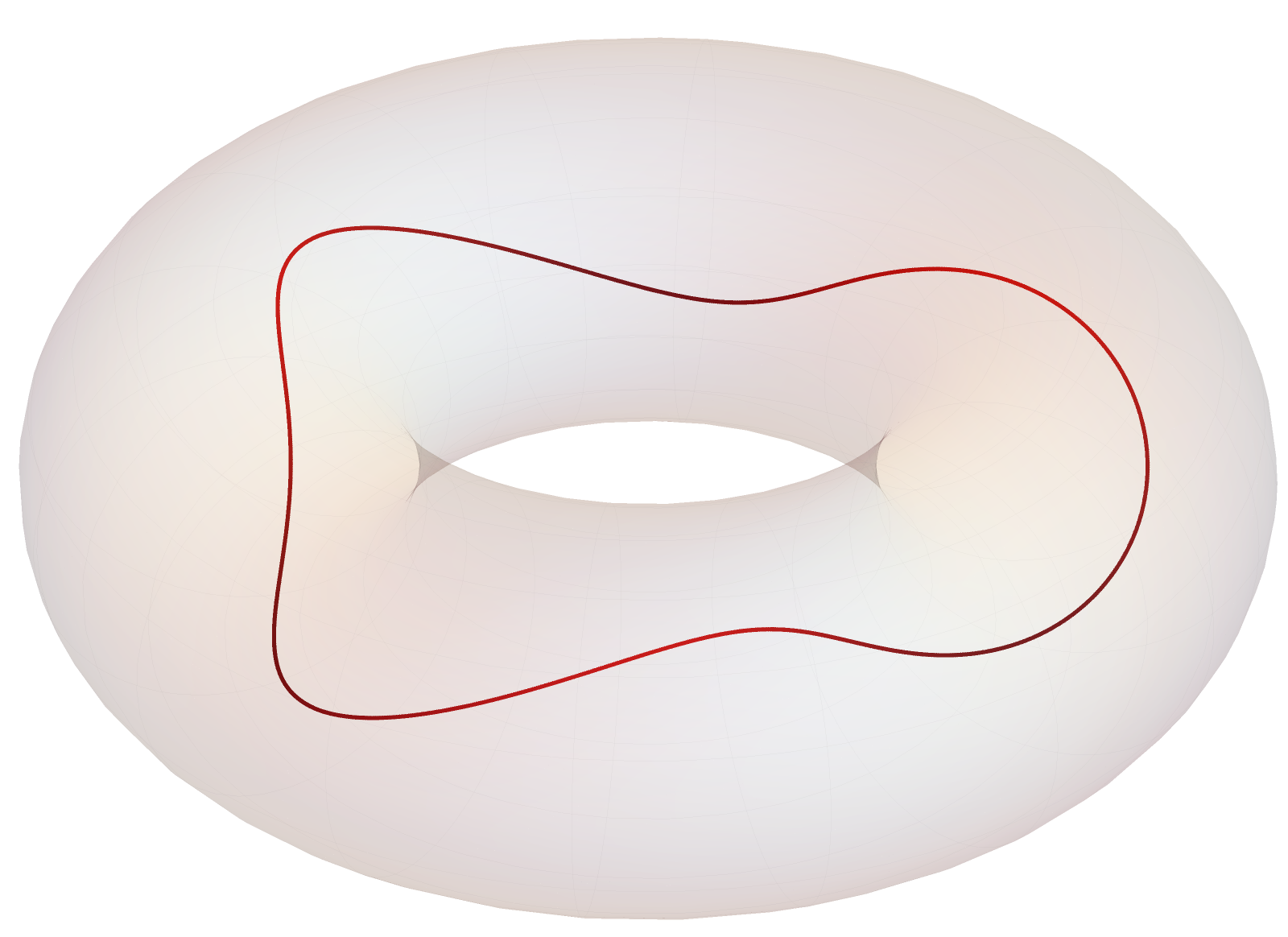}
        \caption{$L_1$}
    \end{subfigure}
    \hfill
    \begin{subfigure}[b]{0.30\textwidth}
        \includegraphics[width=\textwidth]{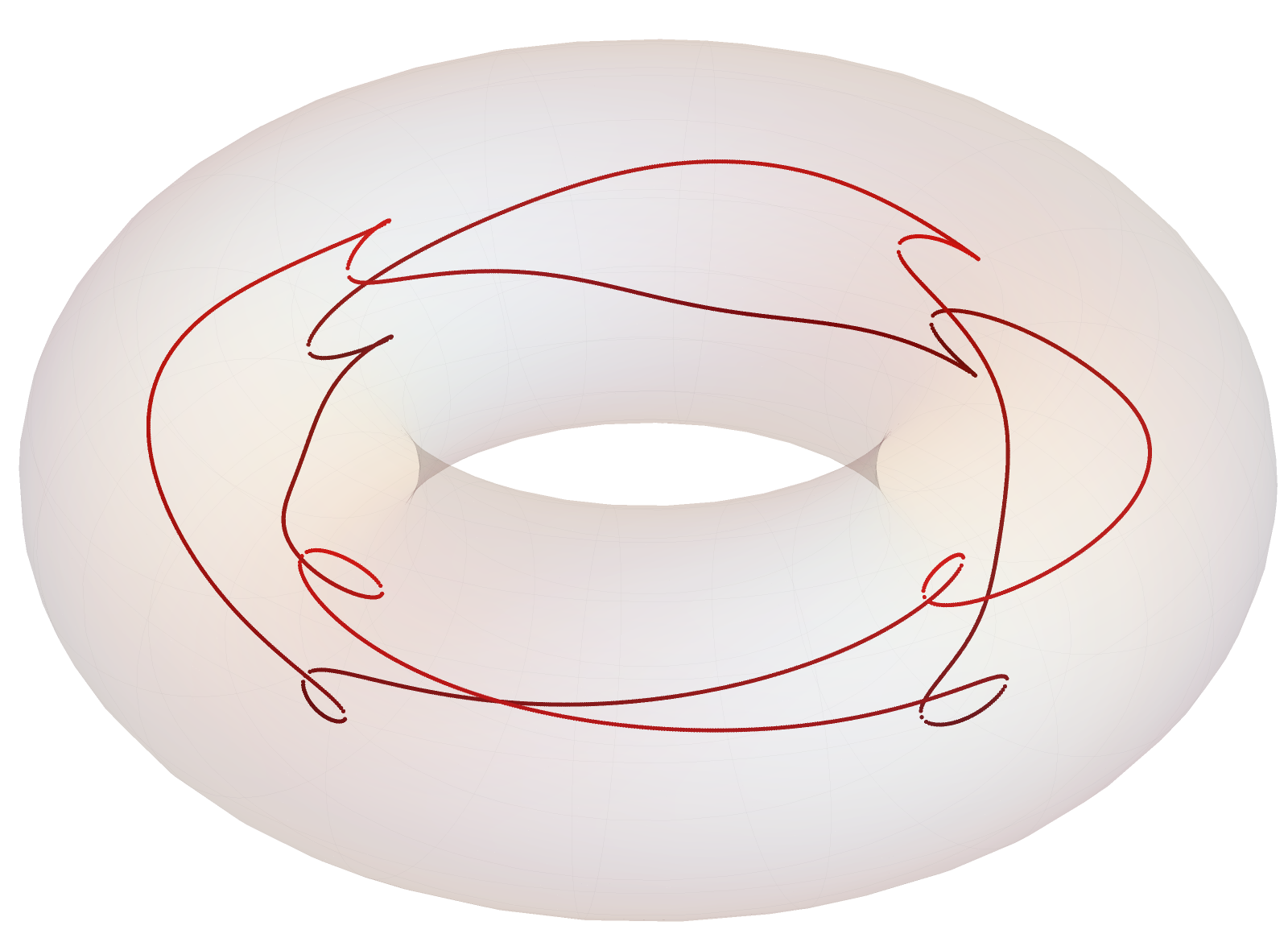}
        \caption{$L_{2}$}
    \end{subfigure}
    \hfill
    \begin{subfigure}[b]{0.30\textwidth}
        \includegraphics[width=\textwidth]{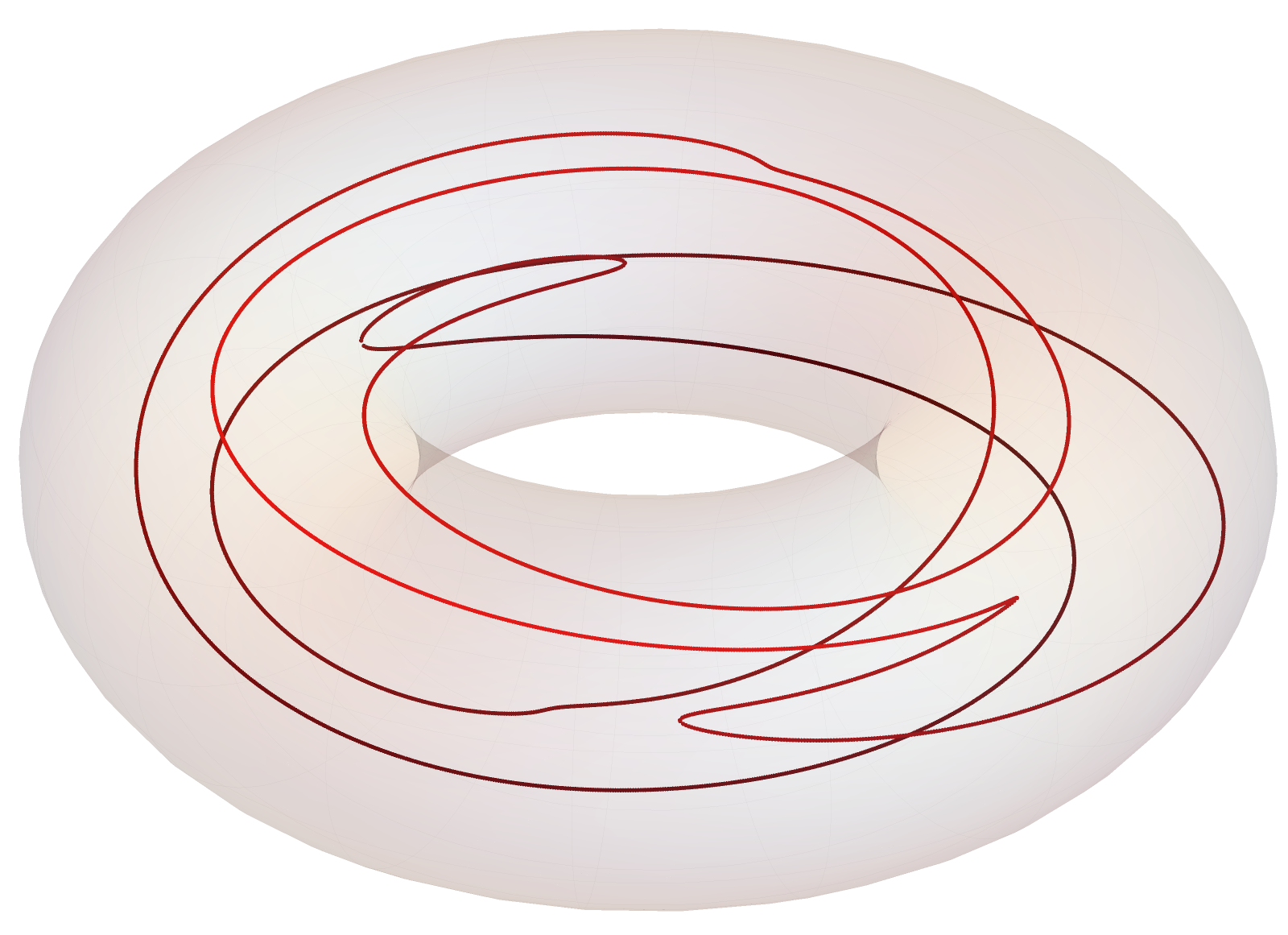}
        \caption{$L_{3}$}
    \end{subfigure}
    \caption{Links associated with $\Delta^1_1, \Delta^1_2$, and $\Delta^1_3$.}
    \label{fig:linksofgis}
\end{figure}
\end{example}

We take $0< \varepsilon_1< \varepsilon_2<\cdots< \varepsilon_{N-1}$ small and define
\[
\begin{aligned}
&X_{1} := \{(u,v) \in S^3 : 0<|u| \leq \varepsilon_1\},\\
&X_{i} := \{(u,v) \in S^3 : \varepsilon_{i-1} < |u| \leq \varepsilon_i\},
\qquad i =2,3,\ldots, N-1,\\
&X_{N} := \{(u,v) \in S^3 :  \varepsilon_{N-1}<|u| <1 \}.
\end{aligned}
\]
Define $\alpha:=\{(u,v) \in S^3: u=0 \}$ and $\beta:=\{(u,v) \in S^3: v=0 \}$. Note that for every $i=1,2,\ldots,N$ there exists a diffeomorphism from the interior $X_i$ to $\mathbb{C}^*\times S^1$ that preserves the arguments of both complex variables.

Since the links $L_i$, $i=1,2,\ldots,N$ are sets of curves in $\mathbb{C}^*\times S^1$, this produces a canonical way in which we can interpret each $L_i$ as a link in the interior of $X_i$.

Note that $S^3=\left(\bigcup_{i=1}^NX_i\right)\cup\alpha\cup\beta$. 
\begin{definition}
We write $\mathbf{L}([L_1,L_2,\ldots,L_{N-1}],L_N)$ 
for the union 
\begin{equation}
    \bigcup_{i=1}^{N}L_i\subset \bigcup_{i=1}^N X_i\subset S^3.
\end{equation}
\end{definition}
The notation $\mathbf{L}([L_1,L_2,\ldots,L_{N-1}],L_N)$ is due to an alternative, but equivalent, description given in \cite{Araujo2024}. Note that the isotopy class of the link $\mathbf{L}([L_1,L_2,\ldots,L_{N-1}],L_N)$ in $S^3$ depends only on the isotopy classes of $L_i$ in $X_i$.

Illustrations for this construction are provided in \cite{Araujo2024, Bode:part2}

The following result is a restatement of \cite[Theorem 1.2]{Araujo2024}, formulated here using an equivalent  description of the link of a convenient and non-degenerate mixed polynomial. 
\begin{theorem}\label{th:linksofinnd}
Let $f:\mathbb{C}^2\to \mathbb{C}$ be a convenient and non-degenerate mixed polynomial. Then, the link of $f$, denoted by $L_f$, is isotopic to the link 
\begin{equation}
\mathbf{L}([L_1, L_2, \dots, L_{N-1}],L_N),
\end{equation}
where each $L_i$ corresponds to the link associated with the respective face of the Newton boundary.
\end{theorem} 
In particular, the link of the singularity $L_f$ consists of various (possibly empty) sublinks $L_i\subset X_i$. We say that a collection of solid tori $V_i$, $i=1,2,\ldots,N-1$, in $S^3$ is \textit{nested} if $V_i\subset V_{i+1}$ for all $i$, and all $V_i$ have a common core. Defining the nested solid tori in $S^3$ as
$$V_i:=\left(\bigcup_{j=1}^{i}X_j\right)\cup\alpha,\ i=1,2,\dots,N-1,$$
we obtain that $\bigcup_{j=1}^i L_j$ is a link in the interior of $V_i$ and $\bigcup_{j=i+1}^N L_i$ is in $S^3\backslash V_i$. Therefore, $\partial V_i$ is an embedded torus in $S^3\backslash L_f$ for every $i=1,2,\ldots,N-1$.  
\begin{example}\label{ex:nestetorii}
Consider the convenient and non-degenerate mixed polynomial
\[
f(u,\bar{u},v,\bar{v}) = u^4 - u^3 v \bar{v} + u^2 v^3 \bar{v}^3 - u v^6 \bar{v}^6 + v^{10} \bar{v}^{10}.
\] 
The maps $g_i$ associated with each 1-face are:
\begin{align*}
g_ 1(u, \rme^{\rmi t}) &= -u+1, \\
g_2(u, \rme^{\rmi t}) &= u^2-u, \\
g_3(u, \rme^{\rmi t}) &= -u^3 +u^2, \\
g_4(u, \rme^{\rmi t}) &= u^4 -u^3. 
\end{align*}
In what follows, we illustrate the components $L_i$ ($i=1, \dots, 4$) of $L_f$ within its nested tori decomposition. 
\end{example}
\begin{figure}[H]
    \centering
    \begin{subfigure}[b]{0.49\textwidth}
        \includegraphics[width=\textwidth]{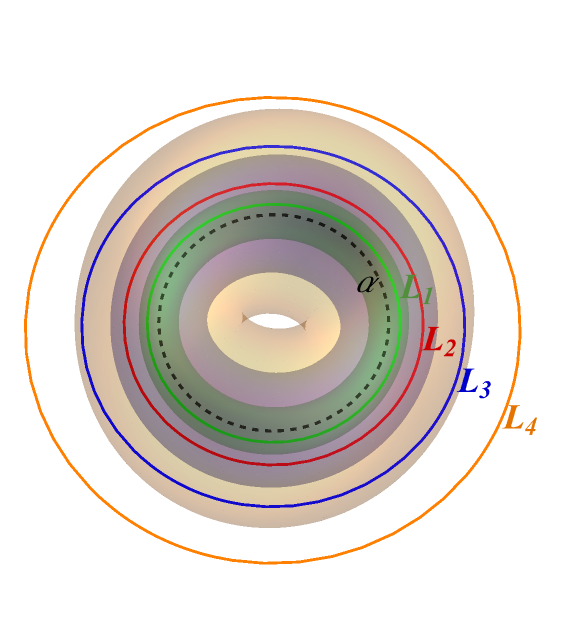}
        \caption{}
    \end{subfigure}
    %\hfill
    \begin{subfigure}[b]{0.49\textwidth}
        \includegraphics[width=\textwidth]{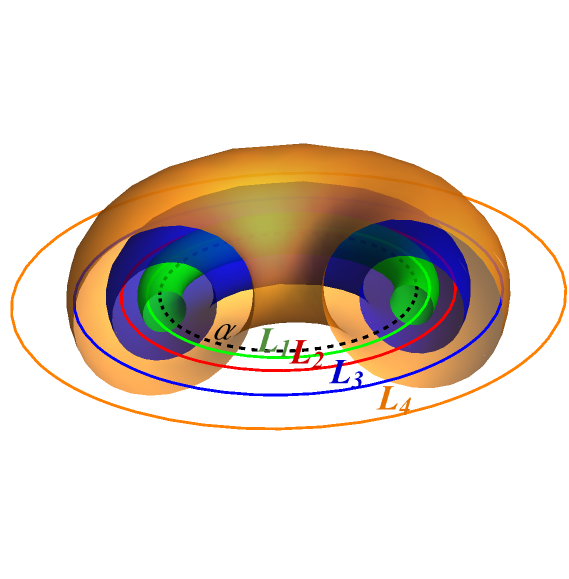}
        \caption{}
    \end{subfigure}
     \caption{Nested solid tori $V_1 \subset V_2 \subset V_3$ sharing a common core $\alpha$. The components $L_i$ ($i=1,\dots,4$) represent the decomposition of the link $L_f= \cup_{i=1}^4 L_i$ according to the Newton boundary.}
       \label{fig:nested}
\end{figure}
\section{Essential Tori in Nested Tori Link Complements}\label{sec:essential}

We begin by clarifying that, unless stated otherwise, all links $L_i$ are assumed to be nonempty. This assumption will be used throughout the section.

Let $L=L_1\cup\cdots\cup L_n$ be a link in the $3$-sphere $S^3$. Let $V_1,\ldots,V_{n-1}$ be
a collection of nested unknotted solid tori with a common core, satisfying
$V_i\subset V_{i+1}$ for $1\le i\le n-2$. Assume that $L_i\subset \text{int}(V_i)$ for all $1\le i\le n-1$, and
that for $2\le i\le {n}$ one has
\[
L_i\subset {S^3}\setminus V_{i-1}.
\]
See Figure~\ref{fig:nested} for an illustration. 
The boundary tori $\partial V_1,\ldots,\partial V_{n-1}$ lie in the exterior of $L$.

Our goal is to determine when the tori $\partial V_1,\ldots,\partial V_{n-1}$ are essential in the
exterior of $L$.

\begin{definition}
Let $V\cong S^1\times D^2$ be a solid torus and let $L\subset \operatorname{int}(V)$ be a
link. The wrapping number of $L$ in $V$ is defined by
\[
\operatorname{wrap}(L)=\min_D |L\cap D|,
\]
where the minimum is taken over all meridional disks $D\subset V$. In other
words, $\operatorname{wrap}(L)$ is the minimal geometric intersection number of $L$ with a
meridional disk.
\end{definition}

\begin{definition}
Let $V\cong S^1\times D^2$ be an oriented solid torus and let $L\subset
\operatorname{int}(V)$ be an oriented link. Choose an oriented meridional disk $D\subset V$. The
winding number of $L$ in $V$ is
\[
w(L)=[L]\cdot [D],
\]
the algebraic intersection number of $L$ with $D$. This is well-defined, i.e.
it does not depend on the choice of meridional disk $D$.
\end{definition}

\begin{remark}
Let $L\subset V$ be a link in a solid torus. Then
\[
|w(L)|\le \operatorname{wrap}(L).
\]
Indeed, the winding number is the algebraic intersection number with a
meridional disk, while the wrapping number is the minimal geometric intersection
number. Since algebraic intersections may cancel, their absolute
value cannot exceed the total number of intersection points.
\end{remark}

\begin{lemma}\label{unknotted}
Let $V_1\subset S^3$ be a solid torus and set
\[
W:=S^3\setminus \operatorname{int}(V_1),
\]
so that $W$ is also a solid torus. If $K\subset W$ is a knot whose wrapping number
in $W$ is $1$ and which is unknotted in $S^3$ (i.e. bounds an embedded disk in $S^3$),
then $K$ is isotopic in $W$ to a meridian of $\partial V_1$.
\end{lemma}

\begin{proof}
Let $K\subset W$ be a knot of wrapping number $1$ in $W$ which is unknotted
in $S^3$. Let $D\subset S^3$ be an embedded disk with $\partial D=K$, and set
\[
C:=D\cap \partial V_1.
\]
Choose $D$ so that $C$ has the minimal possible number of components.

Each component of $C$ is a simple closed curve lying in the interior of the
disk $D$ (since $\partial D=K$ is disjoint from $\partial V_1$). Suppose, for contradiction, that
some component $c\in C$ is inessential on $\partial V_1$. Then $c$ bounds an embedded
disk $E\subset \partial V_1$.

Choose $c$ to be innermost on $D$, i.e., let $\Delta\subset D$ be the subdisk with $\partial \Delta=c$
and $\operatorname{int}(\Delta)\cap C=\varnothing$. Note that $\Delta$ and $E$ meet only along their common
boundary $c$. Replacing $\Delta$ by $E$ (smoothing corners if necessary) produces a
new embedded disk $D'$ with the same boundary $\partial D'=K$ and strictly fewer
intersection curves with $\partial V_1$ than $D$. This contradicts the minimality of $|C|$.

Hence no component of $C$ is inessential on $\partial V_1$, so every component is
essential.

Denote by $K'$ the outermost loop in $C$. This loop is isotopic to $K$ in $W$.

Any essential curve on the torus $\partial V_1$ corresponds to a slope and hence
is isotopic to a torus knot $T(p,q)\subset \partial V_1$ for some coprime integers $p,q$.
Here $T(p,q)$ winds $p$ times in the longitudinal direction and $q$ times in the
meridional direction of $\partial V_1$. The wrapping number of $K$ in $W$ equals $|q|$.
Since $K$ has wrapping number $1$, we conclude that $|q|=1$. But the only torus
knots $T(p,\pm 1)$ are trivial knots on $\partial V_1$, each isotopic in $W$ to a meridian of
$V_1$.

Thus $K$ is isotopic in $W$ to a meridian of $\partial V_1$, completing the proof.
\end{proof}

We begin by characterizing when $\partial V_1$ is an essential torus:

\begin{proposition}\label{Pro1}
\leavevmode
\begin{itemize}
\item If $n\ge 3$, then the torus $\partial V_1$ is essential if and only if:
\begin{itemize}
\item $L_1$ has wrapping number greater than one in $V_1$, or has wrapping
number equal to one and is not the trivial knot; and
\item at least one of $L_2,\ldots,L_n$ has non-zero wrapping number in
$S^3\setminus V_1$.
\end{itemize}

\item If $n=2$, then the torus $\partial V_1$ is essential if and only if:
\begin{itemize}
\item $L_1$ has wrapping number greater than one in $V_1$, or has wrapping
number equal to one and is not the trivial knot; and
\item $L_2$ has wrapping number greater than one in $S^3\setminus V_1$, or has wrapping number equal
to one and is not the trivial knot.
\end{itemize}
\end{itemize}
\end{proposition}

\begin{proof}
Suppose first that $n\ge 3$. If the torus $\partial V_1$ is essential in the exterior
of $L$, then it is incompressible and not boundary parallel in $S^3\setminus L$.

On the side $V_1'$ of $S^3\setminus \partial V_1$ containing $L_1$, incompressibility implies that
$L_1$ cannot be contained in a $3$-ball inside $V_1$. Moreover, since $\partial V_1$ is not
boundary parallel on this side, it follows that $L_1$ is not the core of the solid
torus $V_1'$. Equivalently, $L_1$ has wrapping number greater than one in $V_1$, or
has wrapping number equal to one and is not a trivial knot (i.e. it is not
the core of $V_1'$). Indeed, if $L_1$ has wrapping number one and is trivial, then
$\partial V_1$ is boundary parallel to $\partial N(L_1)$ by Lemma~\ref{unknotted}.

On the side $S^3\setminus V_1'$ containing $L_2,\ldots,L_n$, there are at least two components.
Hence, $\partial V_1$ cannot be boundary parallel to the boundary of a tubular
neighbourhood of a single component. For incompressibility to hold on this
side, at least one of $L_2,\ldots,L_n$ must have non-zero wrapping number in
$S^3\setminus V_1$.

Conversely, assume that $L_1$ has wrapping number greater than one in $V_1$,
or has wrapping number equal to one and is not a trivial knot, and that at
least one of $L_2,\ldots,L_n$ has non-zero wrapping number in $S^3\setminus V_1$. Then $\partial V_1$
is incompressible on both sides and is not boundary parallel. Therefore, $\partial V_1$
is essential in the exterior of $L$.

Now suppose that $n=2$. If $\partial V_1$ is essential, then, as before, on the
side containing $L_1$, incompressibility together with the fact that $\partial V_1$ is not
boundary parallel implies that $L_1$ has wrapping number greater than one in
$V_1$, or has wrapping number equal to one and is not a trivial knot.

On the side containing $L_2$, for $\partial V_1$ to be incompressible and not boundary
parallel, the component $L_2$ cannot have wrapping number zero and cannot be
a trivial knot with wrapping number one. Thus, there are two possibilities:
\begin{itemize}
\item $L_2$ has wrapping number greater than one in $S^3\setminus V_1$, or
\item $L_2$ has wrapping number equal to one but is not a trivial knot.
\end{itemize}

Conversely, if $L_1$ has wrapping number greater than one in $V_1$, or has
wrapping number equal to one and is not a trivial knot, and one of the above
conditions holds for $L_2$, then $\partial V_1$ is incompressible on both sides and is not
boundary parallel. Hence, $\partial V_1$ is essential in the exterior of $L$.
\end{proof} 

\begin{remark}\label{rem1}
Let $V_1\subset V_2\subset\ldots\subset V_{n-1}$ be unknotted nested solid tori and let $L_i\subset V_i$ for all $i<n$ and $L_i\subset S^3\backslash V_{i-1}$ for all $i>1$. Then the wrapping number of $L_i$ in $V_i$ is equal to its wrapping number in $V_j$ for all $j>i$. The same holds for the winding numbers.
\end{remark}

\begin{lemma}\label{lem:additive}
    Let $L_i\subset V_i$ for all $i<n$ and $L_i\subset S^3\backslash V_{i-1}$ for all $i>1$. Then for all $i<n$ the wrapping number of $L_1\cup L_2\cup\ldots\cup L_i$ in $V_i$ is equal to $\sum_{j=1}^{i}\operatorname{wrap}(L_j)$, where $\operatorname{wrap}(L_j)$ denotes the wrapping number of $L_j$ in $V_j$.
\end{lemma}
\begin{proof}
We prove the lemma by induction on $i$. For $i=1$ the statement is trivial. Now suppose that we have the desired equality for some $i<n$. We first show that the wrapping number of the union of the $i+1$ links in $V_{i+1}$ is bounded from below by the sum of the wrapping numbers of the individual links in their corresponding tori.
Consider a meridional disk $D$ of $V_{i+1}$. Its geometric intersection number with $L_{i+1}$ is at least $\operatorname{wrap}(L_{i+1})$. After an isotopy that does not increase the geometric intersection number with $L_{i+1}$ we can assume that $D$ intersects $\partial V_i$ in exactly one closed loop, which must be a meridian of $\partial V_i$, since $V_i$ and $V_{i+1}$ have the same core. (This argument is similar to the proof of Lemma 3.4.) The intersection $D\cap V_i$ is thus a meridional disk of $V_i$ whose geometric intersection number with $L_1\cup L_2\cup\ldots\cup L_i$ is at least its wrapping number, which by the induction hypothesis is equal to $\sum_{j=1}^i\operatorname{wrap}(L_j)$. Thus the geometric intersection number between $D$ and $L_1\cup L_2\cup\ldots\cup L_{i+1}$ in $V_{i+1}$ is at least $\sum_{j=1}^{i+1}\operatorname{wrap}(L_j)$.

Now consider a meridional disk $D$ of $V_{i+1}$ that realizes the wrapping number of $L_{i+1}$. Then as in the previous argument, we can assume that it intersects $\partial V_i$ in a meridian of $\partial V_i$. By replacing $D_1:=D\cap V_i$ with a meridional disk $D_2$ of $V_i$ that realizes the wrapping number of $L_1\cup L_2\cup\ldots\cup L_i$ in $V_i$ we obtain a meridianol disk $D'=(D\backslash D_1)\cup D_2$ whose geometric intersection number with $L_1\cup L_2\cup \ldots\cup L_{i+1}$ is equal to $\operatorname{wrap}(L_1\cup L_2\cup\ldots \cup L_i)+\operatorname{wrap}(L_{i+1})$, which by the induction hypothesis is $\sum_{j=1}^i\operatorname{wrap}(L_j)+\operatorname{wrap}(L_{i+1})=\sum_{j=1}^{i+1}\operatorname{wrap}(L_j)$. Since the wrapping number is the minimal geometric intersection between any meridional disk and the link, we have that the wrapping number of $L_1\cup L_2\cup\ldots\cup L_{i+1}$ in $V_{i+1}$ is at most $\sum_{j=1}^{i+1}\operatorname{wrap}(L_j)$.

The two inequalities together prove the desired equality for $i+1$, which proves the lemma by induction. 
\end{proof}
Note that in general the wrapping number is not additive. We only obtain the equality in Lemma~\ref{lem:additive} because the links do not intersect the boundaries of the nested tori. Figure~\ref{fig:whitehead} shows a Whitehead link in a solid torus such that each individual component has wrapping number zero, but the wrapping number of the two-component link is two.

\begin{figure}[h]
\centering
\includegraphics[height=3cm]{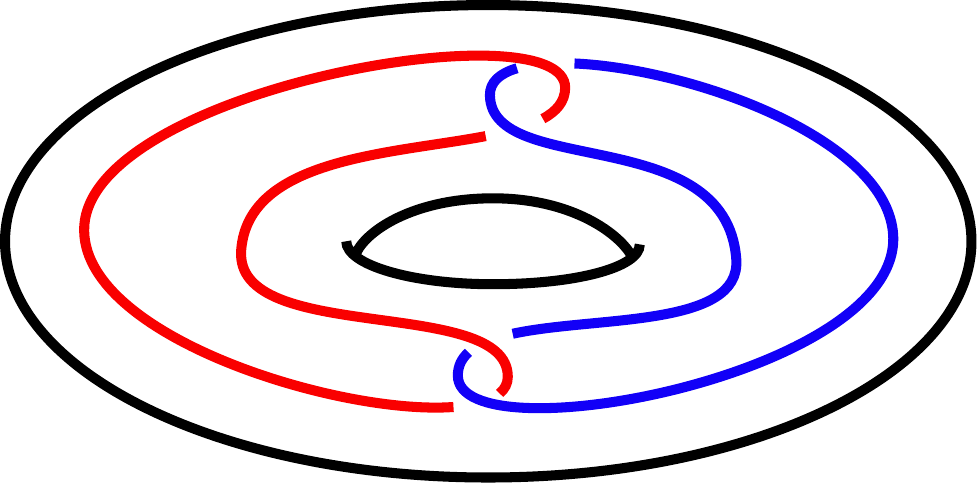}
\caption{A Whitehead link in a solid torus. The wrapping number is not additive in general.\label{fig:whitehead}}
\end{figure}

The winding number on the other hand is always additive.

\begin{remark}\label{rem:symmetries}
Let $V_1\subset V_2\subset\ldots\subset V_{n-1}$ be a set of unknotted nested solid tori and let $W_i=S^3\backslash \text{int}(V_i)$. Then $V'_1:=W_{n-1}\subset V'_2:=W_{n-2}\subset\ldots\subset V'_{n-1}:=W_1$ also form a set of unknotted nested solid tori. If the links satisfy $L_i\subset V_i$ for all $i<n$ and $L_i\subset S^3\backslash V_{i-1}$ for all $i>1$, then the links $L'_1:=L_n$, $L'_2:=L_{n-1},\ldots,L'_n:=L_1$ satisfy $L'_i\subset V'_i$ for all $i<n$ and $L'_i\subset S^3\backslash V'_{i-1}$ for all $i>1$.
\end{remark}

We now characterize when $\partial V_{n-1}$ is an essential torus. We assume $n>2$,
as the case $n=2$ was addressed in the previous proposition.

\begin{proposition}\label{Pro2}
Consider $n>2$. The torus $\partial V_{n-1}$ is essential if and only
if at least one of $L_1,\ldots,L_{n-1}$ has wrapping number greater than zero inside
$V_{n-1}$ and one of the following holds:
\begin{itemize}
\item $L_n$ has wrapping number greater than one in $S^3\setminus V_{n-1}$; or
\item $L_n$ is a nontrivial knot with wrapping number equal to one in
$S^3\setminus V_{n-1}$.
\end{itemize}
\end{proposition}
\begin{proof}
The proposition follows by applying the first case of Proposition~\ref{Pro1}, viewing $L_n$ as lying inside $V'_1$ and $L_1, L_2, \dots, L_{n-1}$ as lying in $S^3 \setminus V'_1$, together with Remarks~\ref{rem1} and \ref{rem:symmetries}.
\end{proof}
We now characterize the conditions under which the tori $\partial V_2,\ldots,\partial V_{n-2}$
are essential. We assume $n>3$, since the cases $n=2$ and $n=3$ were already
treated in the previous propositions.

\begin{proposition}\label{Pro3}
Suppose $n>3$ and let $1<i<n-1$. The torus $\partial V_i$ is
essential if and only if
\begin{itemize}
\item at least one of the components $L_1,\ldots,L_i$ has wrapping number greater
than zero inside $V_i$, and
\item at least one of the components $L_{i+1},\ldots,L_n$ has wrapping number
greater than zero in $S^3\setminus V_i$.
\end{itemize}
\end{proposition}
\begin{proof}
We apply the second item of Proposition~\ref{Pro1} to the solid torus $V'_1:=V_i$ with $L'_1:=L_1\cup L_2\cup\ldots\cup L_i$ and $L'_2:=L_{i+1}\cup\ldots\cup L_n$. Since both $L'_1$ and $L'_2$ have at least two components, they are not trivial knots. Therefore, $V_i$ is essential if and only the wrapping numbers of both $L'_1$ and $L'_2$ in $V_i$ and $S^3\backslash V_i$, respectively, are greater than zero. By Lemma~\ref{lem:additive} this is equivalent to the conditions stated in the proposition.
\end{proof}

\begin{theorem}\label{th:char_essentialtori}
Let $L=L_1\cup\cdots\cup L_n$ be a link in $S^3$, and let $V_1,\ldots,V_{n-1}$ be
nested trivial solid tori with a common core and $V_i\subset V_{i+1}$ for $1\le i\le n-2$,
where $L_i\subset V_i$ and $L_i\subset V_i\setminus V_{i-1}$ for $2\le i\le n-1$. Denote $W_i:=S^3\setminus \operatorname{int}(V_i)$.
The tori $\partial V_1,\ldots,\partial V_{n-1}$ lie in the exterior $E(L)=S^3\setminus N(L)$. Then $\partial V_i$ is
an essential torus in $E(L)$ precisely in the following cases:

\medskip
\noindent
(ii) If $i = 1$:
\begin{itemize}
\item If $n\ge 3$, the torus $\partial V_1$ is essential if and only if:
\begin{itemize}
\item $L_1$ has wrapping number greater than one in $V_1$, or has wrapping
number equal to one and is not the trivial knot; and
\item at least one of $L_2,\ldots,L_n$ has non-zero wrapping number in
$W_1$.
\end{itemize}

\item If $n=2$, $\partial V_1$ is essential if and only if:
\begin{itemize}
\item $L_1$ has wrapping number greater than one in $V_1$, or has wrapping
number equal to one and is not the trivial knot; and
\item $L_2$ has wrapping number greater than one in $W_1$, or has wrapping number equal to one and is not the trivial knot.
\end{itemize}
\end{itemize}   

\medskip
\noindent
(ii) If $1<i<n-1$ and $n>3$: $\partial V_i$ is essential if and only if
\begin{itemize}
\item at least one of $L_1,\ldots,L_i$ has wrapping number greater than zero in $V_i$, and
\item at least one of $L_{i+1},\ldots,L_n$ has wrapping number greater than zero in $W_i$.
\end{itemize}

\medskip
\noindent
(iii) If $i=n-1$ and $n>2$: $\partial V_{n-1}$ is essential if and only if
\begin{itemize}
\item at least one of $L_1,\ldots,L_{n-1}$ has wrapping number greater than zero in $V_{n-1}$, and
\item either $L_n$ has wrapping number greater than one in $W_{n-1}$, or $L_n$ has wrapping
number equal to one in $W_{n-1}$ and is nontrivial in $W_{n-1}$.
\end{itemize}
\end{theorem}

\begin{proof}
The result follows from Propositions~\ref{Pro1}, \ref{Pro2}, and \ref{Pro3}.
\end{proof}

\begin{definition}
Let $L\subset S^3$ be a link. A $2$-sphere $F$ properly embedded in
$S^3\setminus N(L)$ is essential if it does not bound a $3$-ball in $S^3\setminus N(L)$.
\end{definition}

A link $L$ is called irreducible if there is no essential $2$-sphere in
$S^3\setminus N(L)$.

\begin{theorem}\label{thm:n>3}
If $n\ge 4$, then the link $L=L_1\cup\cdots\cup L_n$ contains an
essential sphere or an essential torus in its exterior. In particular, $L$ is not
hyperbolic.
\end{theorem}

\begin{proof}
Assume that $n\ge 4$. We analyze the position of the links $L_i$ with
respect to the solid tori $V_2$ and $S^3\setminus V_2$.

First, consider the link $L_1$. If the wrapping number of $L_1$ in $V_2$ is zero,
then the wrapping number of $L_1$ in $V_1$ is zero as well. Since $L_1$ and $L_2$
are separated by $\partial V_1$, then $L_1$ is contained in a $3$-ball in $V_1$. The
boundary of this ball is a sphere that separates $L_1$ from the remaining link
components of the link in $S^3\setminus V_2$, and hence it is an essential sphere in the
exterior of $L$.

Thus, now we may assume that the wrapping number of $L_1$ in $V_2$ is
non-zero. It follows that $\partial V_2$ is incompressible on the side containing $L_1$.
Moreover, since at least two components of $L$ lie inside $V_2$, the torus $\partial V_2$ is
not boundary parallel on this side.

Next, consider $L_3\subset S^3\setminus V_2$. If the wrapping number of $L_3$ in $S^3\setminus V_2$ is
zero, then $L_3$ is contained in a $3$-ball disjoint from the remaining components
of the link in $V_2$. As before, this yields an essential sphere in the exterior of
$L$.

Now we may assume that the wrapping number of $L_3$ in $S^3\setminus V_2$ is non-zero.
By the same argument, $\partial V_2$ is incompressible and not boundary-parallel on
this side as well.

Hence, either an essential sphere appears in the above cases, or else $\partial V_2$ is
incompressible and not boundary parallel on both sides. In the latter case,
$\partial V_2$ is an essential torus in the exterior of $L$.

Therefore, it follows from Theorem~\ref{Thurston} that $L$ is not hyperbolic.
\end{proof}

\section{Links of Mixed Polynomials and their Complements.}\label{sec:results}

\subsection{Wrapping and winding numbers as multiplicities}
We now return to the specific situation of links of mixed singularities. As seen in Section~\ref{sec:background} there is a natural structure of nested solid tori $V_i$, whose boundaries $\partial V_i$ divide the link $L_f$ into various components $L_i\subset V_i$. In particular, each $\partial V_i$ is an embedded torus in the link complement $S^3\backslash L_f$. In the previous section, we have seen that the essentiality of each $\partial V_i$ can be analyzed using the wrapping numbers of the different links $L_i$. In this section, we develop tools to study the wrapping number of $L_i$ in terms of data from the Newton boundary.

We are interested in estimating the wrapping numbers of the components of $L_f$ within the solid tori $V_i$ and $W_i=S^3\backslash \text{int}(V_i)$. Since the wrapping number of a link is bounded from below by the absolute value of its winding number, controlling the latter is sufficient to guarantee the existence of essential tori in $E(L_f):=S^3\backslash L_f$ via Theorem~\ref{th:char_essentialtori}. By Remark~\ref{rem1} and Lemma~\ref{lem:additive} calculating the wrapping and winding numbers of $L_f$ in $V_i$ and $W_i$ reduces to analyzing the links $L_1\cup L_2\cup\ldots\cup L_i\subset V_i$ and $L_{i+1}\cup L_{i+2}\cup\ldots\cup L_N\subset W_i$.  

The winding number of a link, in contrast to its wrapping number, depends on the orientation of the link. We define an orientation of each $L_i$ using the local structure of $g_i$ around $L_i$. The assumption that $f$ is non-degenerate implies that all points on $L_i$ are regular points of $g_i$. By continuity there exists a tubular neighbourhood $U$ of $L_i$ where $g_i$ is also regular and since the image of $L_i$ is 0, there exists local coordinates $(x,y,\alpha)$ of $U\backslash L_i$ in which the argument map $\arg(g_i)=\tfrac{g_i}{|g_i|}$ takes the form $\arg(g_i)(x,y,\alpha)=\arg(x+\rmi y)=\tfrac{x+\rmi y}{\sqrt{x^2+y^2}}$. Here $\alpha$ corresponds to a coordinate along a component of $L_i$ and $(x,y)$ are the coordinates of a disk that is transverse to $L_i$ with $(x=0,y=0)$ corresponding to the intersection point between $L_i$ and the disk. The behaviour of $\arg{g_i}$ in a neighbourhood of $L_i$ is thus exactly that of an open book (even if $L_i$ is not necessarily a fibered link), see Figure~\ref{fig:openbook}a). The level surfaces of $\arg(g_i)$ in $U_i\backslash L_i$ are annuli with an orientation given by the normal vector $\nabla\arg(g_i)$. Since each component of $L_i$ is part of the boundary of one of these annuli, their orientations induce an orientation on each component of $L_i$ and thus an orientation on $L$.

\begin{figure}[h]
    \centering
    \begin{subfigure}[b]{0.30\textwidth}
        \includegraphics[height=3cm]{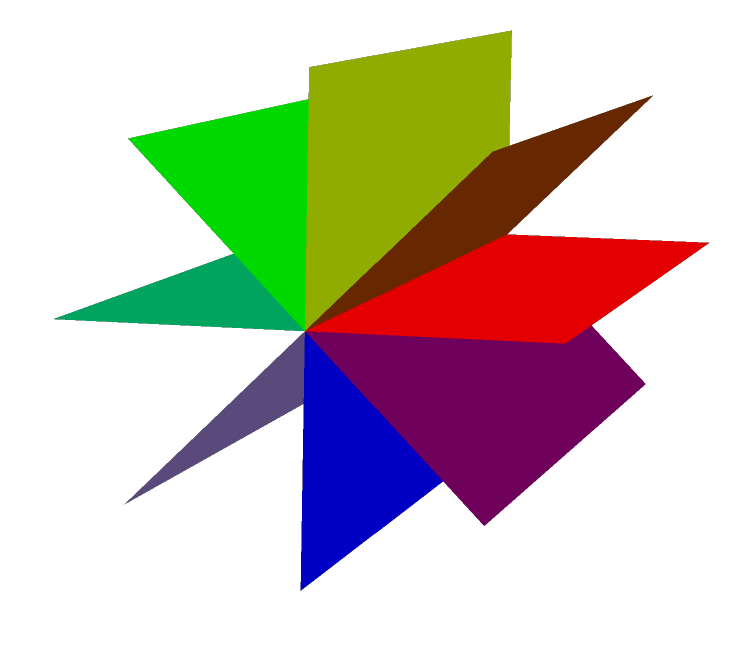}
        \caption{}
    \end{subfigure}
    \begin{subfigure}[b]{0.30\textwidth}
        \includegraphics[height=3cm]{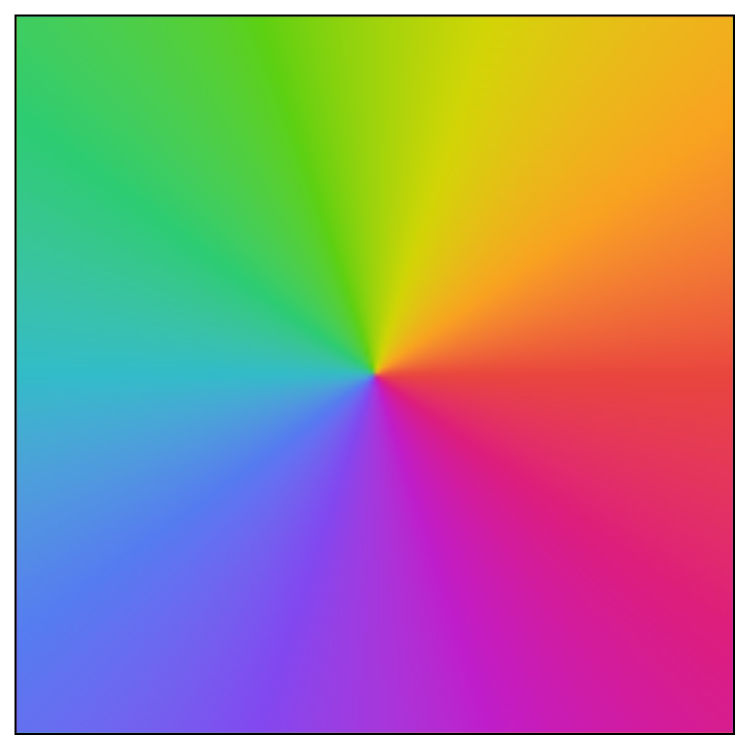}
        \caption{}
    \end{subfigure}
    \begin{subfigure}[b]{0.30\textwidth}
        \includegraphics[height=3cm]{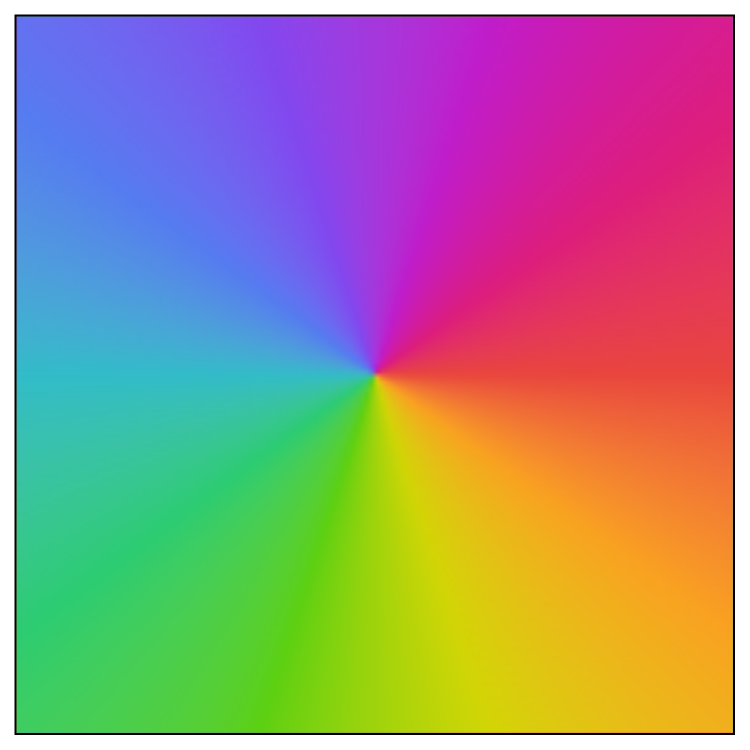}
        \includegraphics[height=3cm]{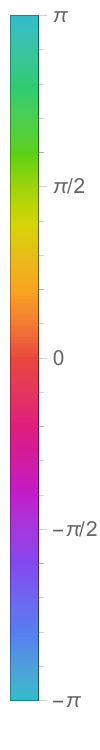}
        \caption{}
    \end{subfigure}
    \caption{(A) The behaviour of $\arg(g_i)$ in a tubular neighbourhood of $L_i$. Each level set is a surface, all of which have the common boundary $L_i$. (B) A positive intersection between $L_i$ and a plane $t=t_*$. The shown colours are the values of $\arg(g_i)(\cdot,t_*)$. The intersection occurs where the different colours meet, since $g_i(L_i)=0$ does not have a well-defined argument. (C) A negative intersection between $L_i$ and a plane $t=t_*$.}
    \label{fig:openbook}
\end{figure}

Now consider a fixed value of $t_*\in[0,2\pi]$ and the restriction of $g_i(u,\rme^{\rmi t})$ to the plane $t=t_*$. Assume that the intersections between $L_i$ and this plane are transverse. Then we can deduce the sign of each intersection point (with respect to the orientation defined above) from the winding of $\arg(g_i)$ around the intersection point, see Figure~\ref{fig:openbook}b) and c). Going along a small loop around an intersection point we encounter (by the discussion above) all values of $\arg(g_i)$ between 0 and $2\pi$ either strictly increasing (corresponding to a positive intersection) or strictly decreasing (corresponding to a negative intersection). In other words, the topological degree of $\arg(g_i(u, \rme^{\rmi t_*}))|_{\partial D_\varepsilon}$ with $D_\varepsilon$ being a disk of small radius $\varepsilon$ centered at the intersection point is either $+1$ or $-1$. The sign of the degree is exactly the sign of the intersection with respect to the orientation of $L_i$ defined above. 

For a link $L$ in $V_i$ we denote its winding number in $V_i$ by $w(L)$ and for a link $L$ in $W_i=S^3\backslash \text{int}(V_i)$ we denote its winding number in $W_i$ by $w'(L)$. %{\color{red}BB:maybe introduce this notation already in earlier sections.}
Recall that for every mixed polynomial $f:\mathbb{C}^2\to\mathbb{C}$ with $N$ 1-faces in its Newton boundary, we have defined functions $g_i:\mathbb{C}\times S^1\to\mathbb{C}$ as in \eqref{eq:def_gi}. The links $L_i$ are the zeros of $g_i$ with non-zero $u$-coordinate. By abuse of notation we also write $L_i$ for the corresponding link in $V_i\backslash V_{i-1}$, identifying $\mathbb{C}^*\times S^1$ with $V_i\backslash V_{i-1}$ (or $\mathbb{C}^*\times S^1$ with $V_1\backslash\alpha$ in the case of $i=1$) as in Section~\ref{sec:background}. In particular, the wrapping and winding numbers of $L_i$ in $V_i$ are equal to the wrapping and winding numbers of $L_i$ in the domain of $g_i$, that is, $\mathbb{C}\times S^1$.

For every fixed $t$ and $\varphi$ we define the mixed univariate polynomials  $g_i^t(u, \bar{u}) := g_i(u, \rme^{\rmi t})$.
For every root $\alpha\in\mathbb{C}$ of a univariate mixed polynomial $g:\mathbb{C}\to\mathbb{C}$ we define its multiplicity with sign $\mathrm{m}_{\mathrm{s}}(g,\alpha)$ as the mapping degree (also called rotation number) of the function 
 \begin{equation}\label{eq:argumentfunction}
     \frac{g}{|g|}:S^1_{\varepsilon}(\alpha) \to S^1,  z \mapsto \frac{g(u,\bar{u})}{|g(u,\bar{u})|}
 \end{equation}
for a sufficiently small $\varepsilon$ and $S^1_{\varepsilon}(\alpha):=\{z \in \mathbb{C} : |z-\alpha|=\varepsilon\}.$ For a simple root, we have $\mathrm{m}_{\mathrm{s}}(g,\alpha)=\pm 1$. If it is $+1$ (resp. $-1$), $\alpha$ is said $\textit{positive simple root}$ (resp. \textit{negative simple root}).

Note that the algebraic intersection number of $L_i\subset V_i$ with the section $\mathbb{C}\times\{\rme^{\rmi t}\}$ is equal to the sum of the multiplicities with sign of the non-zero roots of $g_i^t$. In other words,
\begin{equation}\label{eqwindingnumber}
w(L_i)=\sum_{\alpha \in V(g_i^t)\setminus \{0\}} \mathrm{m}_{\mathrm{s}}(g_i^t,\alpha),
\end{equation}
This value in \eqref{eqwindingnumber} does not depend on $t$ because the winding number is the algebraic intersection number between the link and any meridional disk. 

Note that  
\begin{equation}\label{eq:multCminuszero}
\sum_{\alpha \in V(g_i^t)\setminus \{0\}} \mathrm{m}_{\mathrm{s}}(g_i^t,\alpha)=\mathrm{m}_{\mathrm{s}, \mathrm{tot}}(g_i^t)-\mathrm{m}_{\mathrm{s}}(g_i^t,0),
\end{equation}
where \(\mathrm{m}_{\mathrm{s}, \mathrm{tot}}(g_i^t):= \sum_{\alpha \in V(g_i^t)} \mathrm{m}_{\mathrm{s}}(g_i^t,\alpha)\) is 
\textit{the total multiplicity with sign.}
Observe that $\mathrm{m}_{\mathrm{s}, \mathrm{tot}}(g_i^t)$ can also be interpreted as the degree of the restriction of $\tfrac{g_i^t}{|g_i^t|}$ to a very large circle.

Analogously to $g_i$ and $g_i^t$ we can define
\begin{equation}
    g_{\underline{i}}(\rme^{\rmi \varphi},v)=f_{\Delta^1_i}(\rme^{\rmi \varphi},\rme^{-\rmi \varphi},v,\bar{v}).\label{eq:def_gunderi}
\end{equation}
 and for every fixed $\varphi$ we define $g_{\underline{i}}^\varphi(v, \bar{v}) := g_{\underline{i}}(\rme^{\rmi \varphi}, v)$. Note that the zeros of $g_{\underline{i}}$ with non-zero $v$-coordinate are (again by convenience, $\Gamma$-niceness and non-degeneracy of $f$, see Lemma~\ref{lemma:Li}) compact subsets of $S^1\times\mathbb{C}^*$. 

\begin{lemma}\label{lem:isotopic}
Let $L_{\underline{i}}$ be the set of zeros of $g_{\underline{i}}:S^1\times\mathbb{C}\to \mathbb{C}$ with non-zero $v$-coordinate and let $L_i$ be the set of zeros $g_i:\mathbb{C}\times S^1\to\mathbb{C}$ with non-zero $u$-coordinate. Then, considering both $S^1\times\mathbb{C}^*$ and $\mathbb{C}^*\times S^1$ as a thickened torus, $L_{\underline{i}}$ and $L_i$ are isotopic.
\end{lemma} 
\begin{proof}
The face function $f_{\Delta_i^1}$ is radially weighted homogeneous, i.e., there exist positive integers $p_1,p_2$ and $d$ such that $f(\lambda^{p_1}u,\lambda^{p_2}v)=\lambda^df(u,v)$ for all $(u,v)\in\mathbb{C}^2$ and all $\lambda\in\mathbb{R}$. 

Consider $S^3\backslash(\alpha\cup\beta)$, where $\alpha=\{(u,v)\in S^3:u=0\}$, $\beta=\{(u,v)\in S^3:v=0\}$. This space is diffeomorphic to $\mathbb{C}^*\times S^1$ via the diffeomorphism $h_1:S^3\backslash(\alpha\cup\beta)\to\mathbb{C}^*\times S^1$, $h_1(u,r\rme^{\rmi t})=(r^{-p_1/p_2}u,\rme^{\rmi t})$. Since $f_{\Delta_i^1}$ is radially weighted homogeneous, the diffeomophism $h_1$ sends $V(f_{\Delta_i^1})\cap (S^3\backslash\{\alpha\cup\beta\})$ to $V(f_{\Delta_i^1})\cap (S^1\times\mathbb{C}^*)=L_i$ (simply take $\lambda=r^{-1/p_2}$).\\

Similarly, the diffeomorphism $h_2:S^3\backslash(\alpha\cup\beta)\to S^1\times \mathbb{C}^*$, $h_2(R\rme^{\rmi\varphi},v)=(\rme^{\rmi \varphi},R^{-p_2/p_1}v)$ sends $V(f_{\Delta_i^1})\cap (S^3\backslash\{\alpha\cup\beta\})$ to $V(f_{\Delta_i^1})\cap (\mathbb{C}^*\times S^1)=L_{\underline{i}}$ (simply take $\lambda=R^{-1/p_1}$).\\

The composition $h_1\circ h_2^{-1}$ is an orientation-preserving diffeomorphism of the thickened torus to itself, mapping $L_{\underline{i}}$ to $L_i$. Since it preserves the arguments of the two complex coordinates, $h_1\circ h_2^{-1}$ is isotopic to the identity map on the thickened torus, and therefore the two links are isotopic.
\end{proof}
 
Note that in the construction of $L_f$ in Section~\ref{sec:background} where we identify $X_i$ with $\mathbb{C}^*\times S^1$, we could just as easily identify $X_i$ with $S^1\times\mathbb{C}^*$ and by Lemma~\ref{lem:isotopic} obtain the same link. Independent of its ambient space ($\mathbb{C}^*\times S^1$, $S^1\times \mathbb{C}$ or $X_i$, all of which are thickened tori), we will write $L_i$ for the zeros of $g_i$ with non-zero $u$-coordinate, the zeros of $g_{\underline{i}}$ with non-zero $v$-coordinate or the corresponding curves in $X_i$. In particular, the wrapping and winding numbers of $L_i$ in $W_{i-1}$ are equal to the wrapping and winding numbers of the zeros of $g_{\underline{i}}$ with non-zero $v$-coordinate in $S^1\times \mathbb{C}$.

We can thus calculate the algebraic intersection number of $L_{i}$ with a section $\{\rme^{\rmi \varphi}\} \times \mathbb{C}$ by \begin{equation}\label{eqwindingnumber2}
w'(L_i):=\sum_{\alpha \in V(g_{\underline{i}}^\varphi)\setminus \{0\}} \mathrm{m}_{\mathrm{s}}(g_{\underline{i}}^\varphi,\alpha)=\mathrm{m}_{\mathrm{s}, \mathrm{tot}}(g_{\underline{i}}^\varphi)-\mathrm{m}_{\mathrm{s}}(g_{\underline{i}}^\varphi,0).
\end{equation}
Analogously to $w(L_i)$, this value is independent of $\varphi \in [0, 2\pi]$.

The dimension of the intersection $V(g^t_i) \cap \mathbb{C}^*$ is typically $0$ (the generic case), but can be $\pm 1$ (defining the dimension of the empty set to be -1). If $V(g^t_i) = \emptyset$, the total multiplicity with sign is defined as $0$ by convention. In this case, the multiplicity with sign at the origin is $0$ because the map \eqref{eq:argumentfunction} has a vanishing topological degree, consistent with the absence of roots.

By varying $t$ we can
 assume the generic case ($\text{dim}=0$), thus $V(g^t_i)\setminus \{0\}$ is a finite set of roots. 
By the non-degeneracy of $f$, the roots of $g_i^t$ in $\mathbb{C}^*$ are simple (positive or negative). The total number of these roots is given by:
\begin{equation}\label{eq:numberofroot}
\begin{aligned}
\text{Number of roots of $g_i^t$ in } \mathbb{C}^* &= \left| \mathrm{m}_{\mathrm{s,tot}}(g^t_i) - \mathrm{m}_{\mathrm{s}}(g^t_i, 0) \right| + 2K \\
&= |w(L_i)|+2K, \qquad \qquad K \in \mathbb{N}_0.
\end{aligned}
\end{equation}
The term $K$ accounts for pairs of roots in $\mathbb{C}^*$ with opposite sign. While such pairs contribute zero to the algebraic sum  $\mathrm{m}_{\mathrm{s,tot}}(g^t_i) - \mathrm{m}_{\mathrm{s}}(g^t_i, 0)$ they count as two distinct points in the total number of roots in $\mathbb{C}^*$. Consequently, the term $2K$ in \eqref{eq:numberofroot} compensates for these ``cancelled" pairs.
Analogously, the total number of the roots of $g_{\underline{i}}^\varphi$ in $\mathbb{C}^*$ is given by
\begin{equation}\label{eq:numberofrootgbari}
\begin{aligned}
\text{Number of roots of $g_{\underline{i}}^\varphi$ in } \mathbb{C}^* &= \left| \mathrm{m}_{\mathrm{s,tot}}(g_{\underline{i}}^\varphi) - \mathrm{m}_{\mathrm{s}}(g_{\underline{i}}^\varphi, 0) \right| + 2K' \\
&= |w'(L_i)|+2K', \qquad \qquad K' \in \mathbb{N}_0.
\end{aligned}
\end{equation}

Let $h(u,\bar{u})$ be a mixed univariate polynomial. It is called \textit{homogeneous of degree $d$} if it satisfies the following identity:
 \[h(\lambda u, \lambda \bar{u})=\lambda^d h( u, \bar{u}), \ \lambda \in \mathbb{R}_{>0}. \]  

For any mixed univariate polynomial
\[h(u,\bar{u})=\sum_{\nu,\mu} c_{\nu,\mu} u^\nu \bar{u}^{\mu}\]
and every positive integer $\ell$, we define 
\[h_{\ell}(u,\bar{u})=\sum_{\nu+\mu=\ell} c_{\nu,\mu} u^\nu \bar{u}^{\mu}.\]
By definition, $h_{\ell}$ is a mixed homogeneous polynomial of degree $\ell$. Analogously to the complex case, any non-zero mixed univariate polynomial admits a decomposition into its non-zero homogeneous parts, that is, 
\[h(u,\bar{u})=h_{\underline{d}}(u,\bar{u})+\dots +h_{\bar{d}}(u,\bar{u}),\]
where  $\underline{d}:=\min\{\nu+\mu : c_{\nu,\mu}\neq 0\}$ and $\bar{d}:=\max\{\nu+\mu : c_{\nu,\mu}\neq 0\}$. The integers $\underline{d}$ and $\bar{d}$ are called \textit{minimal degree} and \textit{maximal degree} of $h$, respectively.
\begin{lemma}\label{lemma:windingnumber}
Let $f \colon \mathbb{C}^2 \to \mathbb{C}$ be a convenient and non-degenerate mixed polynomial that is $\Gamma$-nice. For each 1-face $\Delta^1_i \in \mathcal{P}(f)$, let $L_i$  be the associated links. The winding numbers of $L_i$ relative to $V_i$ and $W_{i-1}$ are given by:
\begin{equation}\label{eq:windingLi}
\begin{aligned}
    w(L_i) &= \mathrm{m}_{\mathrm{s}}(f^t_{\Delta_i}, 0) - \mathrm{m}_{\mathrm{s}}(f^t_{\Delta_{i-1}}, 0), \\
    w'(L_i) &= \mathrm{m}_{\mathrm{s}}(f^\varphi_{\Delta_{i-1}}, 0) - \mathrm{m}_{\mathrm{s}}(f^\varphi_{\Delta_i}, 0),
\end{aligned}
\end{equation}
where $\Delta_i \in \mathcal{P}_0(f)$. Furthermore, for each $i \in \{1,2, \dots, N\}$, the winding numbers relative to  $V_i$ and $W_{i-1}$ of the following sublinks of $L_f$ satisfy:
\begin{equation}\label{eq:windingnestedlinks}
\begin{aligned}
    w(L_1\cup L_2\cup\dots\cup L_i) &= \mathrm{m}_{\mathrm{s}}(f^t_{\Delta_i}, 0), \\
    w'(L_i\cup L_{i+1}\cup\ldots\cup L_N) &= \mathrm{m}_{\mathrm{s}}(f^\varphi_{\Delta_{i-1}}, 0).
\end{aligned}
\end{equation}
\end{lemma}
\begin{proof}
We have
\[
(g_i^t)_{\underline{d}}(u,\bar{u})
= f_{\Delta_{i-1}}(u,\bar{u},\rme^{\rmi t},\rme^{-\rmi t})
\quad\text{and}\quad
(g_i^t)_{\bar{d}}(u,\bar{u})
= f_{\Delta_i}(u,\bar{u},\rme^{\rmi t},\rme^{-\rmi t}).
\]
For \(i=0,1,\dots,N\), set
\begin{equation}\label{eq:tpoly}
f^t_{\Delta_i}(u,\bar{u})
:= f_{\Delta_i}(u,\bar{u},\rme^{\rmi t},\rme^{-\rmi t}).
\end{equation}
 Since $f$ is $\Gamma$-nice, convenient and non-degenerate, by Remark~\ref{rmk:niceness}(iii), we get that \((g_i^t)_{\bar{d}}\)  and  \((g_i^t)_{\underline{d}}\) have no roots in $\mathbb{C}^{*}$. Hence, 
by \cite[Theorem~20]{Oka2012}, 
\begin{equation}\label{eq:invms0total_i}
\begin{split}
\mathrm{m}_{\mathrm{s},\mathrm{tot}}(g_i^t)
&= \mathrm{m}_{\mathrm{s}}((g_i^t)_{\bar{d}},0)= \mathrm{m}_{\mathrm{s}}(f^t_{\Delta_i},0), \\
\mathrm{m}_{\mathrm{s}}(g_i^t,0)
&=\mathrm{m}_{\mathrm{s}}((g_i^t)_{\underline{d}},0)= \mathrm{m}_{\mathrm{s}}(f^t_{\Delta_{i-1}},0).
\end{split}
\end{equation}
Therefore, by \eqref{eqwindingnumber} and \eqref{eq:multCminuszero}, we obtain the desired equality for \(w(L_i)\).

Analogously, 
we set 
\begin{equation}\label{eq:varphipoly}
f^\varphi_{\Delta_i}(v,\bar{v})
:= f_{\Delta_i}(\rme^{\rmi \varphi},\rme^{-\rmi \varphi},v,\bar{v}).
\end{equation}
We have
\[
(g_{\underline{i}}^\varphi)_{\underline{d}}(v,\bar{v})
= f^\varphi_{\Delta_i}(v,\bar{v})
\quad \text{and} \quad
(g_{\underline{i}}^\varphi)_{\bar{d}}(v,\bar{v})
= f^\varphi_{\Delta_{i-1}}(v,\bar{v}).
\]
Again by Remark~\ref{rmk:niceness}(iii) and \cite[Theorem~20]{Oka2012},
\begin{equation}\label{eq:invms0total_underline_i}
\begin{split}
\mathrm{m}_{\mathrm{s},\mathrm{tot}}(g_{\underline{i}}^\varphi)
&= \mathrm{m}_{\mathrm{s}}(f^\varphi_{\Delta_{i-1}},0),\\
\mathrm{m}_{\mathrm{s}}(g_{\underline{i}}^\varphi,0)
&= \mathrm{m}_{\mathrm{s}}(f^\varphi_{\Delta_i},0).
\end{split}
\end{equation}
Therefore, by \eqref{eqwindingnumber2}, we obtain the desired equality for \(w'(L_i)\). 

To get \eqref{eq:windingnestedlinks}, note that, by \eqref{eq:windingLi}:
$$ w(L_1\cup L_2\cup\dots\cup L_i) = \sum_{j=1}^i w(L_j)=\mathrm{m}_{\mathrm{s}}(f^t_{\Delta_i}, 0)-\mathrm{m}_{\mathrm{s}}(f^t_{\Delta_0}, 0)  $$ and  
$$w'(L_i\cup L_{i+1}\cup\ldots\cup L_N)=\sum_{j=i}^Nw(L_j)= \mathrm{m}_{\mathrm{s}}(f^\varphi_{\Delta_{i-1}},0)-\mathrm{m}_{\mathrm{s}}(f^\varphi_{\Delta_N},0).$$
Due to the convenience and non-degeneracy of $f$ the mixed polynomials $f_{\Delta_0}$ and $f_{\Delta_N}$ are non-degenerate and semiholomorphic. Specifically, $f_{\Delta_0}$ is independent of $u$ and $\bar{u}$, while $f_{\Delta_N}$ is independent of $v$ and $\bar{v}$. Consequently, by Remark~\ref{rmk:niceness}(iii), $f^t_{\Delta_0}$ and $f^\varphi_{\Delta_N}$ are non-zero constant mixed univariate polynomials. Since they lack roots at the origin, it follows that:
$$\mathrm{m}_{\mathrm{s}}(f^t_{\Delta_0}, 0) = \mathrm{m}_{\mathrm{s}}(f^\varphi_{\Delta_N}, 0) = 0.$$ This implies the last part of the lemma.
\end{proof}
Note that in Section~\ref{sec:essential} we made the assumption that the different links $L_i$ in $V_i$ are all nonempty, while in the construction of $L_f$ in Section~\ref{sec:background} some of them could be empty. In order to apply the results from Section~\ref{sec:essential} we should thus determine if the link $L_i$ associated with a 1-face of the Newton boundary of a mixed polynomial is nonempty.

%{\color{red}(BB: I don't think the next result is really a proposition. We're saying that if $L_i$ has non-zero wrapping, then $L_i$ is not empty. Maybe it should be stated in terms of multiplicities instead. Then it becomes a less trivial.}
\begin{lemma}\label{cor:nonemptycriterion}
Let \(f:\mathbb{C}^2 \to \mathbb{C}\) be a convenient and non-degenerate mixed polynomial that is $\Gamma$-nice, and fix \(i \in \{1,2,\dots,N\}\). If
\[
\max\bigl\{|\mathrm{m}_{\mathrm{s}}(f^t_{\Delta_i}, 0) - \mathrm{m}_{\mathrm{s}}(f^t_{\Delta_{i-1}}, 0)|,\,|\mathrm{m}_{\mathrm{s}}(f^\varphi_{\Delta_{i-1}}, 0) - \mathrm{m}_{\mathrm{s}}(f^\varphi_{\Delta_i}, 0)|\bigr\} > 0,
\]
then \(L_i\) is nonempty.
\end{lemma}
\begin{proof}
By Lemma~\ref{lemma:windingnumber}, we get  $|w(L_i)|>0$ or $|w'(L_i)|>0$.  Consequently, by \eqref{eq:numberofroot} and \eqref{eq:numberofrootgbari}, $L_i$ is nonempty. 
\end{proof}
\begin{remark}
  The above result provides sufficient conditions for \(L_i\) to be nonempty; however, these are not necessary conditions to get it. Indeed, consider the mixed polynomial
\[
f(u,\bar{u},v,\bar{v})=(u-v)\overline{(u+v)},
\]
for which \(w(L_1)=w'(L_1)=0\), while \(L_1\) is nonempty. The winding number can be computed by observing that the links associated with \(u-v\) and \(\overline{(u+v)}\) are distinct longitudes with opposite orientations. 
\end{remark}
We define $I_f:=\{i \in \{1,2,\dots, N\} \mid L_{i} \text{ is nonempty}\}$. 

\begin{remark}
Let $h(u, \bar{u})$ be a mixed univariate polynomial with $u = x_1 + \rmi x_2$. We consider the Wirtinger derivatives $\frac{\partial h}{\partial u}$ and $\frac{\partial h}{\partial \bar{u}}$ that are defined through the partial derivatives $\frac{\partial h}{\partial x_1}$ and $\frac{\partial h}{\partial x_2}$ as:
$$\frac{\partial h}{\partial u} = \frac{1}{2} \left( \frac{\partial h}{\partial x_1} - \rmi \frac{\partial h}{\partial x_2} \right) \quad \text{and} \quad \frac{\partial h}{\partial \bar{u}} = \frac{1}{2} \left( \frac{\partial h}{\partial x_1} + \rmi \frac{\partial h}{\partial x_2} \right).$$
By \cite[Proposition~15]{Oka2012}, if \(h\) is a mixed univariate polynomial, then a simple root \(\alpha\) is positive (resp. negative) if and only if
\[
\left|\frac{\partial h}{\partial u}(\alpha)\right|
-
\left|\frac{\partial h}{\partial \bar{u}}(\alpha)\right|
>0
\quad
(\text{resp. }<0).
\]
Consequently, if \(g_i\) is \(u\)- (resp. \(\bar{u}\)-) semiholomorphic, then every simple root of \(g_i^t\) is positive (resp. negative). The same applies to \(g_{\underline{i}}^\varphi\) under the assumption of being \(v\)- (resp. \(\bar{v}\)-) semiholomorphic. It follows that, under these assumptions, the absolute value of the winding number and the wrapping number of \(L_i\) in \(\mathbb{C} \times S^1\) and \(L_i\) in \(S^1 \times \mathbb{C}\) coincide with the number of roots of \(g_i^t\) and \(g_{\underline{i}}^\varphi\) in \(\mathbb{C}^*\), respectively.
\end{remark}
\subsection{Main results}
%\begin{theorem}
%Let \(f:\mathbb{C}^2 \to \mathbb{C}\) be a convenient, non-degenerate mixed polynomial that is \(\Gamma\)-nice, and let 
%\(\mathcal{P}_0(f)=\{\Delta_0, \Delta_1, \dots, \Delta_N\}\). Fix \(i \in \{1,2,\dots,N-1\}\). If
%\[
%\min \Bigl\{ \bigl|\mathrm{m}_{\mathrm{s}}(f^{t}_{\Delta_i},0)\bigr|, 
%\bigl|\mathrm{m}_{\mathrm{s}}(f^{\varphi}_{\Delta_i},0)\bigr| \Bigr\} > 1,
%\]
%then \(\partial V_i\) is essential. 
%\end{theorem}
\begin{proof}[Proof of Theorem~\ref{prop:fastcriterion}]
By Lemma~\ref{lemma:windingnumber},
\[
w(L_1\cup L_2\cup\dots\cup L_i) = \mathrm{m}_{\mathrm{s}}(f^t_{\Delta_i},0).
\]
It follows that
\[
\bigl|w(L_1\cup L_2\cup\dots\cup L_i)\bigr| = \Bigl|\sum_{j=1}^i w(L_j)\Bigr| > 1.
\] 
Hence, there exists some \(j \le i\) such that \(|w(L_j)| > 1\).  

Similarly, by Lemma~\ref{lemma:windingnumber},
\[
w'(L_{i+1}\cup L_{i+2}\cup\ldots\cup L_N) = \mathrm{m}_{\mathrm{s}}(f^\varphi_{\Delta_i},0).
\]
It follows that
\[
\bigl|w'(L_{i+1}\cup L_{i+2}\cup\ldots\cup L_N)\bigr| = \Bigl|\sum_{j=i+1}^N w(L_{j})\Bigr| > 1.
\]  
Hence, there exists some \(k > i\) such that \(|w'(L_k)| > 1\).  

Therefore, by Theorem~\ref{th:char_essentialtori}, we obtain that  \(\partial V_i\) is essential.
\end{proof}

Note that, while our criterion can be used to prove $w(L_j)>1$ for some $j\leq i$, the property \(|w(L_j)| > 1\) for some \(j \le i\) is not equivalent to \(|w(L_1\cup L_2\cup\dots\cup L_i)| > 1\).  
For example, if \(i=2\), \(w(L_1)=2\) and \(w(L_2)=-2\), then \(w(L_1\cup L_2)=0\).  

Since each \(w(L_i)\) can be computed using Lemma~\ref{lemma:windingnumber}, this allows us to formulate, based on Theorem~\ref{th:char_essentialtori}, more refined conditions to identify essential tori in \(E(L_f)\).  

In what follows, we assume \(\#(I_f) \ge 3\); the case \(\#(I_f)=2\) is covered by Theorem~\ref{prop:fastcriterion}.

%\begin{theorem}
%Let $f:\mathbb{C}^2 \to \mathbb{C}$ be a convenient and non-degenerate mixed polynomial that is $\Gamma$-nice. Let $\mathcal{P}_0(f)=\{\Delta_0, \Delta_1,\dots, \Delta_N\}$ and  $I_f=\{i_1,i_2,\dots,i_{n}\}$, $n\geq 3$. 

%\medskip
%\noindent (i) If $|\mathrm{m}_{\mathrm{s}}(f^{t}_{\Delta_{i_1}},0)|>1$ and for some $i_j>i_1$, $$|\mathrm{m}_{\mathrm{s}}(f^\varphi_{\Delta_{i_j-1}},0)-\mathrm{m}_{\mathrm{s}}(f^\varphi_{\Delta_{i_j}},0)|>0,$$ then $\partial V_{i_1}$ is essential. 
    
%     \medskip
%\noindent
%(ii) If $1<k<n-1$ and for $i_j\leq i_k<i_l$, $$\min(|\mathrm{m}_{\mathrm{s}}(f^{t}_{\Delta_{i_j}},0)-\mathrm{m}_{\mathrm{s}}(f^{t}_{\Delta_{i_j-1}},0)|,|\mathrm{m}_{\mathrm{s}}(f^\varphi_{\Delta_{i_l-1}},0)-\mathrm{m}_{\mathrm{s}}(f^\varphi_{\Delta_{i_l}},0)|)>0,$$ 
%then $\partial V_{i_k}$ is essential.

%    \medskip
%\noindent
%(iii) If $|\mathrm{m}_{\mathrm{s}}(f^\varphi_{\Delta_{i_n}},0)|\}>1$ and for some $i_j<i_n$, $$|\mathrm{m}_{\mathrm{s}}(f^{t}_{\Delta_{i_j}},0)-\mathrm{m}_{\mathrm{s}}(f^{t}_{\Delta_{i_j-1}},0)|>0,$$ then $\partial V_{i_{n-1}}$ is essential.
%\end{theorem}
\begin{proof}[Proof of Theorem~\ref{thm:geralcriterionessentialtori}]\noindent (i) By the hypotheses of the theorem and the winding number formulas established in Lemma~\ref{lemma:windingnumber}, we have $|w(L_1\cup L_2\cup\dots\cup L_{i_1}])| > 1$. Since the component $L_{i_1}$ is the unique nonempty link within this configuration, the total winding number is concentrated in this single component, yielding $|w(L_{i_1})|=|w(L_1\cup L_2\cup\dots\cup L_{i_1}])| > 1$. Furthermore, the existence of an index $i_j > i_1$ such that $|w'(L_{i_j})| > 0$ is guaranteed by the formulas in Lemma~\ref{lemma:windingnumber}. Combining these non-vanishing winding numbers, the criteria of Theorem~\ref{th:char_essentialtori}(i) are satisfied, allowing us to conclude that $\partial V_{i_1}$ is an essential torus.

\vspace{0.2cm}\noindent (ii) Here, the hypotheses ensure the existence of at least three nonempty components in the topological decomposition given by the nested tori $V_i$, that is, at least three links $L_i, \ i=1,2,\dots,N,$ are nonempty. Applying these hypotheses to the formulas in Lemma~\ref{lemma:windingnumber}, we obtain the non-zero winding conditions $|w(L_{i_j})| > 0$ and $|w'(L_{i_l})| > 0$. These conditions fulfill the requirements of Theorem~\ref{th:char_essentialtori}(ii). Consequently, it follows that $\partial V_{i_k}$ is essential.

\vspace{0.2cm}\noindent (iii) The result for this case follows by a symmetric argument to that of part (i), where the roles of the indices are interchanged as dictated by the structure of the Newton boundary.\end{proof}

\begin{proof}[Proof of Theorem~\ref{thm:nonemptycriterion}]
   By Lemma~\ref{cor:nonemptycriterion}, we have $\#(I_f) \geq 4$. Thus, by Theorem~\ref{thm:n>3}, the exterior of $L_f$ contains either an essential sphere or an essential torus, which means that $L_f$ is reducible or toroidal. Therefore, by Corollary~\ref{cor:essential_torus_obstruction} and Theorem~\ref{Thurston}, we conclude that $L_f$ is not hyperbolic.
\end{proof}

\section{Calculations and Examples}\label{sec:examples}
We now review results from \cite{Oka2012} that provide a method for computing 
\(\mathrm{m}_{\mathrm{s}}(f^t_{\Delta_i},0)\) and 
\(\mathrm{m}_{\mathrm{s}}(f^\varphi_{\Delta_i},0)\) for 
\(\Delta_i \in \mathcal{P}_0(f)\). 
First, observe that these mixed univariate polynomials are homogeneous because $f_{\Delta_i}$ is the face function of a vertex in the Newton boundary.

Let $h(u,\bar{u})$ be a mixed univariate polynomial. It is defined to have a \textit{bi-degree} $(m,n)$ where $m=\deg_u h$ and $n=\deg_{\bar{u}}h$. 

For a non-zero complex number $\varsigma$, we define 
\begin{equation}
\epsilon(\varsigma)=\begin{cases}  1 \ \ &|\varsigma|<1 \\
0 \ \  &|\varsigma|=1\\
-1 \ \ & |\varsigma|>1.
\end{cases}
\end{equation}
The following lemma characterizes the total multiplicity of homogeneous mixed univariate polynomials. In our context, this applies to the mixed univariate polynomials $f^\varphi_{\Delta_i}$ and $f^t_{\Delta_i}$, which are homogeneous and do not vanish on $\mathbb{C}^*$ (Remark~\ref{rmk:niceness}(iii)). 
 %{\color{red}(BB: comment about how our polynomials satisfy this condition?)}
\begin{lemma}\label{multhomog}
Let $h(u,\bar{u})$ be a homogeneous mixed univariate polynomial of degree $d$ and bi-degree $(m,n)$ without roots in $\mathbb{C}^*$. Then 
\begin{equation*}
    h(u,\bar{u})=cu^{d-n}\bar{u}^{d-m}\prod_{i=1}^{m+n-d}(u-\varsigma_{i}\bar{u}),  \ c \in \mathbb{C}^*, |\varsigma_{i}| \neq 1, \text{ and }
\end{equation*} 
\begin{equation}\label{eq.multwithsigns}
    \mathrm{m}_{\mathrm{s},\mathrm{tot}}(h)=\mathrm{m}_{\mathrm{s}}(h,0)=m-n+\sum_{i=1}^{m+n-d} \epsilon(\varsigma_{i}).
\end{equation}
\end{lemma}
\begin{proof}
If $h$ is constant, it must be a non-zero constant of degree $d=0$ and bi-degree $(0,0)$. In this case, the result follows trivially. Suppose now that $h$ is non-constant. Since $h$ has bi-degree $(m,n)$, we have $m=\deg_{u}h$, $n=\deg_{\bar{u}}h$ and 
\begin{equation}\label{eq:multhomog1}
h(u,\bar{u})=cu^{d-n}\bar{u}^{d-m} \hat{h}(u,\bar{u}),\ c\in \mathbb{C}^*,
\end{equation}
where $\hat{h}$ is a homogeneous mixed polynomial of degree $m+n-d$ and bi-degree $(m+n-d,m+n-d)$.
Hence, $\hat{h}$ can be expressed as:
\[\hat{h}(u,\bar{u})=\bar{u}^{m+n-d}\hat{h}\left(\frac{u}{\bar{u}},1\right).\]
The polynomial $\hat{h}\left(\frac{u}{\bar{u}},1\right)$, considered in variable $\frac{u}{\bar{u}}$, has $m+n-d$ non-zero roots  $\varsigma_i$ (which may not all be distinct). Since $h$ does not vanish in $\mathbb{C}^*$, we must have $|\varsigma_i|\neq 1$ for all $i$. This leads to the product form:
\begin{equation}\label{eq:multhomog2}
\hat{h}(u,\bar{u})=\bar{u}^{m+n-d} \prod_{i=1}^{m+n-d}\left(\frac{u}{\bar{u}}-\varsigma_i\right)=\prod_{i=1}^{m+n-d}(u-\varsigma_i\bar{u}), \ |\varsigma_i|\neq 1.
\end{equation}
Substituting \eqref{eq:multhomog2} into \eqref{eq:multhomog1} yields the desired decomposition. 

Assume that  $$|\varsigma_1|\leq |\varsigma_2|\leq \cdots \leq |\varsigma_{\ell}|< 1 < |\varsigma_{\ell+1}|\leq \cdots \leq |\varsigma_{m+n-d}|.$$
Consider the family of mixed univariate polynomials in the parameter $s \in [0,1]$:
  $$ h_s(u,\bar{u})=cu^{d-n}\bar{u}^{d-m}u^\ell \bar{u}^{m+n-d-\ell}\prod_{i=1}^{\ell}\left(1-s\varsigma_{i}\frac{\bar{u}}{u}\right)\prod_{i=\ell+1}^{m+n-d}\left(s\frac{u}{\bar{u}}-\varsigma_{i}\right).$$

 It is clear that $h_s$ provides a homotopy:
 \[\frac{h_s}{|h_s|}:S^1_{\varepsilon}(0) \to S^1,\]
 for a sufficiently small $\varepsilon>0$. Therefore, the multiplicity with signs at the origin of $h$ coincides with the  multiplicity  with signs at the origin of $cu^{d-n}\bar{u}^{d-m}u^\ell \bar{u}^{m+n-d-\ell}$, which is clearly equal to
$$m-n+\sum_{i=1}^{m+n-d} \epsilon(\varsigma_{i}).$$
 
 The total multiplicity with signs $\mathrm{m}_{\mathrm{s},\mathrm{tot}}(h)$ and the multiplicity with signs at the origin $\mathrm{m}_{\mathrm{s}}(h,0)$ coincide because $h$ is homogeneous and does not vanish in $\mathbb{C}^*$. Therefore, we obtain the formula in \eqref{eq.multwithsigns}.
\end{proof}

\begin{remark}\label{rmk:bidegree}
\hspace{2.0cm} 

\begin{itemize}
    \item[(i)] If a homogeneous mixed univariate polynomial  $h$ has no roots in $\mathbb{C}$, it is a non-zero constant by homogeneity. In this case, $h$ defines a map \eqref{eq:argumentfunction} with vanishing topological degree, thus $m_s(h,0)=0$.
    \item[(ii)] If a mixed univariate polynomial $h$ has bi-degree $(m,n)$ with $m>n$, the total multiplicity with signs is generically greater than $m-n$. This property implies that the zero set $V(h)$ has dimension zero and is nonempty. For example, in the case where $m>1$ and $n=1$, the zero set $V(h)$ is always nonempty (\cite{Elkadi2017}).
\end{itemize}
\end{remark}
\begin{remark}[Semiholomorphic polynomials]\label{rmk:semiholomorphic}
   Let \(\Delta_i \in \mathcal{P}_0(f)\). If \(f_{\Delta_i}\) is \(u\)- (or \(\bar{u}\)-) semiholomorphic, then \(f^t_{\Delta_i}\) is a mixed univariate polynomial of degree \(\deg_u f^t_{\Delta_i}\) (resp. \(\deg_{\bar{u}} f^t_{\Delta_i}\)) and bi-degree \((\deg_u f^t_{\Delta_i},0)\) (resp. \((0,\deg_{\bar{u}} f^t_{\Delta_i})\)).  
By \eqref{eq.multwithsigns}, if $f^t_{\Delta_i}$ has an isolated root at the origin, we have
\[
\mathrm{m}_{\mathrm{s}, \mathrm{tot}}(f^t_{\Delta_i}) = \mathrm{m}_{\mathrm{s}}(f^t_{\Delta_i},0) = 
\begin{cases}
\deg_u f^t_{\Delta_i}, & \text{if } f_{\Delta_i} \text{ is } u\text{-semiholomorphic},\\[2mm]
-\deg_{\bar{u}} f^t_{\Delta_i}, & \text{if } f_{\Delta_i} \text{ is } \bar{u}\text{-semiholomorphic}.
\end{cases}
\]

Analogously, if \(f_{\Delta_i}\) is \(v\)- (or \(\bar{v}\)-) semiholomorphic, and  $f^\varphi_{\Delta_i}$ has an isolated root at the origin, then
\[
\mathrm{m}_{\mathrm{s}, \mathrm{tot}}(f^\varphi_{\Delta_i}) = \mathrm{m}_{\mathrm{s}}(f^\varphi_{\Delta_i},0) =
\begin{cases}
\deg_v f^\varphi_{\Delta_i}, & \text{if } f_{\Delta_i} \text{ is } v\text{-semiholomorphic},\\[1mm]
-\deg_{\bar{v}} f^\varphi_{\Delta_i}, & \text{if } f_{\Delta_i} \text{ is } \bar{v}\text{-semiholomorphic}.
\end{cases}
\]
\end{remark}

\begin{example}\label{ex:example2}
  Consider the convenient and non-degenerate mixed polynomial
\[
f(u,\bar{u},v,\bar{v}) = u^4 + \bar{u} u^2v +u^2  \bar{v}^2  + v^6.
\]  
The vertices of the Newton boundary are 
\(\Delta_0 = (0,6), \ \Delta_1 = (2,2),\) and \(\Delta_2 = (4,0)\), 
with corresponding face functions:
\[
f_{\Delta_0}(u,\bar{u},v,\bar{v}) = v^6, \quad 
f_{\Delta_1}(u,\bar{u},v,\bar{v})= u^2 \bar{v}^2 , \quad 
f_{\Delta_2}(u,\bar{u},v,\bar{v}) = u^4.
\]
Thus, we get the mixed univariate polynomials associated to the vertices:
\begin{equation*}
\begin{aligned}
    f^t_{\Delta_0}(u,\bar{u}) &= \mathrm{e}^{6 \mathrm{i} t}, & f^{\varphi}_{\Delta_0}(v,\bar{v}) &= v^6, \\
    f^t_{\Delta_1}(u,\bar{u}) &= u^2\mathrm{e}^{-2 \mathrm{i} t}, & f^{\varphi}_{\Delta_1}(v,\bar{v}) &= \mathrm{e}^{2 \mathrm{i} \varphi} \bar{v}^2, \\
    f^t_{\Delta_2}(u,\bar{u}) &= u^4, & f^{\varphi}_{\Delta_3}(v,\bar{v}) &= \mathrm{e}^{4 \mathrm{i} \varphi}.
\end{aligned}
\end{equation*}
Following Remarks~\ref{rmk:semiholomorphic} and \ref{rmk:bidegree}(i), we can directly compute the values in Tables~\ref{tbl:1ex1} and \ref{tbl:2ex1}.
\begin{table}[h!]
\centering
\begin{tabular}{|c|c|c|}
\hline
\(i\) & \(\mathrm{m}_{\mathrm{s}}(f^t_{\Delta_i},0)\) & 
\(\mathrm{m}_{\mathrm{s}}(f^t_{\Delta_i},0) - \mathrm{m}_{\mathrm{s}}(f^t_{\Delta_{i-1}},0)\) \\
\hline
0 & 0 & --- \\
1 & 2 & 2 \\
2 & 4 & 2 \\
\hline
\end{tabular}
\caption{Signed multiplicities and differences for \(f^t_{\Delta_i}\).}
\label{tbl:1ex1}
\end{table}

\begin{table}[h!]
\centering
\begin{tabular}{|c|c|c|}
\hline
\(i\) & \(\mathrm{m}_{\mathrm{s}}(f^\varphi_{\Delta_i},0)\) & 
\(\mathrm{m}_{\mathrm{s}}(f^\varphi_{\Delta_{i-1}},0) - \mathrm{m}_{\mathrm{s}}(f^\varphi_{\Delta_i},0)\) \\
\hline
0 & 6 & --- \\
1 & -2 & 8 \\
2 & 0 & -2 \\
\hline
\end{tabular}
\caption{Signed multiplicities and differences for \(f^\varphi_{\Delta_i}\).}
\label{tbl:2ex1}
\end{table}
By Theorem~\ref{prop:fastcriterion}, we conclude that \(\partial V_1\) is essential. Moreover, $L_f$ is not hyperbolic. 
\end{example}
\begin{example}\label{ex:ex2}
Consider the convenient and non-degenerate mixed polynomial in Example~\ref{ex:example1}:
\[
f(u,\bar{u},v,\bar{v}) = u^5 + u^2 \bar{u}^2 v + u^3 v^2 - \mathrm{i} u \bar{u}^2 v^2 + u^2 \bar{u} v^2 + \bar{u} v^6 + v^9.
\]
The vertices of the Newton boundary are 
\(\Delta_0 = (0,9), \ \Delta_1 = (1,6), \ \Delta_2 = (3,2)\), and $ \Delta_3 = (5,0)$, with associated mixed univariate polynomials:
\begin{equation*}
\begin{aligned}
    f^t_{\Delta_0}(u,\bar{u}) &= \mathrm{e}^{9 \mathrm{i} t}, & f^{\varphi}_{\Delta_0}(v,\bar{v}) &= v^9, \\
    f^t_{\Delta_1}(u,\bar{u}) &= \bar{u}\mathrm{e}^{6 \mathrm{i} t}, & f^{\varphi}_{\Delta_1}(v,\bar{v}) &= \mathrm{e}^{- \mathrm{i} \varphi} v^6, \\
    f^t_{\Delta_2}(u,\bar{u}) &= (u^3 - \mathrm{i} u \bar{u}^2 + u^2 \bar{u}) \mathrm{e}^{2 \mathrm{i} t}, & f^{\varphi}_{\Delta_2}(v,\bar{v}) &= (\mathrm{e}^{3 \mathrm{i} \varphi} - \mathrm{i} \mathrm{e}^{- \mathrm{i} \varphi} + \mathrm{e}^{\mathrm{i} \varphi}) v^2, \\
    f^t_{\Delta_3}(u,\bar{u}) &= u^5, & f^{\varphi}_{\Delta_3}(v,\bar{v}) &= \mathrm{e}^{5 \mathrm{i} \varphi}.
\end{aligned}
\end{equation*}
By Remarks~\ref{rmk:semiholomorphic} and \ref{rmk:bidegree}(i), we can directly compute the entries in the second column of Tables~\ref{tbl:1} and \ref{tbl:2}, with the exception of 
\(\mathrm{m}_{\mathrm{s}}(f^t_{\Delta_2},0)\), which requires Lemma~\ref{multhomog}.  
Indeed, the mixed univariate polynomial $f^t_{\Delta_2}$ is homogeneous of degree \(d=3\) and bi-degree \((m,n) = (3,2)\).  
Applying \eqref{eq.multwithsigns}, we get:
\[
\mathrm{m}_{\mathrm{s}}(f^t_{\Delta_2},0) = 3 - 2 + \sum_{i=1}^{2} \epsilon(\varsigma_i),
\]
where
\[
\varsigma_1 = \mathrm{i}\frac{1+\sqrt{5}}{2}, \qquad
\varsigma_2 = \mathrm{i}\frac{1-\sqrt{5}}{2}.
\]
Since \(|\varsigma_1| > 1\) and \(|\varsigma_2| < 1\), we obtain 
\(\mathrm{m}_{\mathrm{s}}(f^t_{\Delta_2},0) = 1\).
\begin{table}[h!]
\centering
\begin{tabular}{|c|c|c|}
\hline
\(i\) & \(\mathrm{m}_{\mathrm{s}}(f^t_{\Delta_i},0)\) & 
\(\mathrm{m}_{\mathrm{s}}(f^t_{\Delta_i},0) - \mathrm{m}_{\mathrm{s}}(f^t_{\Delta_{i-1}},0)\) \\
\hline
0 & 0 & --- \\
1 & -1 & -1 \\
2 & 1 & 2 \\
3 & 5 & 4 \\
\hline
\end{tabular}
\caption{Signed multiplicities and differences for \(f^t_{\Delta_i}\).}
\label{tbl:1}
\end{table}
\begin{table}[h!]
\centering
\begin{tabular}{|c|c|c|}
\hline
\(i\) & \(\mathrm{m}_{\mathrm{s}}(f^\varphi_{\Delta_i},0)\) & 
\(\mathrm{m}_{\mathrm{s}}(f^\varphi_{\Delta_{i-1}},0) - \mathrm{m}_{\mathrm{s}}(f^\varphi_{\Delta_i},0)\) \\
\hline
0 & 9 & --- \\
1 & 6 & 3 \\
2 & 2 & 4 \\
3 & 0 & 2 \\
\hline
\end{tabular}
\caption{Signed multiplicities and differences for \(f^\varphi_{\Delta_i}\).}
\label{tbl:2}
\end{table}

By Lemma~\ref{cor:nonemptycriterion}, the link \(L_i,\ i=1,2,3,\) is nonempty, so \(I_f = \{1,2,3\}\).  
Hence, applying Theorem~\ref{thm:geralcriterionessentialtori}, we conclude that \(\partial V_2\) is essential. Moreover, $L_f$ is not hyperbolic. 
\\
Note that in this example, Theorem~\ref{prop:fastcriterion} does not allow us to conclude the existence of an essential torus because $|\mathrm{m}_{\mathrm{s}}(f_{\Delta_1}^t,0)|=|\mathrm{m}_{\mathrm{s}}(f_{\Delta_2}^t,0)|=1$.
\end{example}
\begin{example}\label{ex:nonhipn>3}
 Consider the convenient and non-degenerate mixed polynomial in Example~\ref{ex:nestetorii}:
\[
f(u,\bar{u},v,\bar{v}) = u^4 - u^3 v \bar{v} + u^2 v^3 \bar{v}^3 - u v^6 \bar{v}^6 + v^{10} \bar{v}^{10}.
\]  
The vertices of the Newton boundary are 
\(\Delta_0 = (0,20), \ \Delta_1 = (1,12), \ \Delta_2 = (2,6), \ \Delta_3 = (3,2),\) and \(\Delta_4 = (4,0)\).
The mixed univariate polynomials associated with these vertices are given by:
\begin{equation*}
\begin{aligned}
    f^t_{\Delta_0}(u,\bar{u}) &= 1, & f^{\varphi}_{\Delta_0}(v,\bar{v}) &= v^{10}\bar{v}^{10}, \\
    f^t_{\Delta_1}(u,\bar{u}) &=- u, & f^{\varphi}_{\Delta_1}(v,\bar{v}) &= -\mathrm{e}^{\mathrm{i} \varphi} v^6\bar{v}^6, \\
    f^t_{\Delta_2}(u,\bar{u}) &= u^2, & f^{\varphi}_{\Delta_2}(v,\bar{v}) &= \mathrm{e}^{2\mathrm{i} \varphi} v^3\bar{v}^3, \\
    f^t_{\Delta_3}(u,\bar{u}) &=- u^3, & f^{\varphi}_{\Delta_3}(v,\bar{v}) &= -\mathrm{e}^{3\mathrm{i} \varphi} v\bar{v}, \\
    f^t_{\Delta_4}(u,\bar{u}) &= u^4, & f^{\varphi}_{\Delta_4}(v,\bar{v}) &= \mathrm{e}^{4\mathrm{i} \varphi}.
\end{aligned}
\end{equation*}
By a direct application of Lemma~\ref{multhomog} and Remark~\ref{rmk:bidegree}(i), we obtain the values presented in Tables~\ref{tbl:1ex_new} and~\ref{tbl:2ex_new}.
\begin{table}[H]
\centering
\begin{tabular}{|c|c|c|}
\hline
\(i\) & \(\mathrm{m}_{\mathrm{s}}(f^t_{\Delta_i},0)\) & \(\mathrm{m}_{\mathrm{s}}(f^t_{\Delta_i},0) - \mathrm{m}_{\mathrm{s}}(f^t_{\Delta_{i-1}},0)\) \\
\hline
0 & 0 & --- \\
1 & 1 & 1 \\
2 & 2 & 1 \\
3 & 3 & 1 \\
4 & 4 & 1 \\
\hline
\end{tabular}
\caption{Signed multiplicities and differences for \(f^t_{\Delta_i}\).}
\label{tbl:1ex_new}
\end{table}
\begin{table}[H]
\centering
\begin{tabular}{|c|c|c|}
\hline
\(i\) & \(\mathrm{m}_{\mathrm{s}}(f^\varphi_{\Delta_i},0)\) & \(\mathrm{m}_{\mathrm{s}}(f^\varphi_{\Delta_{i-1}},0) - \mathrm{m}_{\mathrm{s}}(f^\varphi_{\Delta_i},0)\) \\
\hline
0 & 0 & --- \\
1 & 0 & 0 \\
2 & 0 & 0 \\
3 & 0 & 0 \\
4 & 0 & 0 \\
\hline
\end{tabular}
\caption{Signed multiplicities and differences for \(f^\varphi_{\Delta_i}\).}
\label{tbl:2ex_new}
\end{table}
Since all values in Table~\ref{tbl:2ex_new} are zero, neither Theorem~\ref{prop:fastcriterion} nor Theorem~\ref{thm:geralcriterionessentialtori} can be applied to conclude that $L_f$ is not hyperbolic. Nevertheless, a direct verification shows that, by Theorem~\ref{thm:nonemptycriterion}, $L_f$ is not hyperbolic. 
\end{example}

\Addresses
\end{document}